\documentclass{amsart}
\usepackage{amstext}
\usepackage{amssymb}
\usepackage{latexsym}
\usepackage{amsmath}
\usepackage{amscd}
\usepackage{amsthm}
\newtheorem{theorem}{Theorem}
\newtheorem{proposition}[theorem]{Proposition}
\newtheorem{lemma}[theorem]{Lemma}

\newtheorem{conjecture}[theorem]{Conjecture}
\newtheorem{remark}[theorem]{Remark}
\newtheorem{definition}[theorem]{Definition}

\newtheorem{problem}[theorem]{Problem}

\begin{document}
\title{Quadratic sequences of powers and Mohanty's conjecture}
\author{Natalia Garcia-Fritz}
\address{Department of Mathematics, University of Toronto\newline 
\indent 40 St. George St. BA6103\newline
\indent Toronto, ON Canada, M5S 2E4}
\email{natalia.garciafritz@utoronto.ca}
\thanks{Part of this work was obtained in the author's thesis at Queen's University, and was partially supported by a Becas Chile Scholarship. This work was completed at the University of Toronto, thanks to the support of a postdoctoral fellowship}
\date{\today}

\dedicatory{To the memory of Jerzy Browkin.}

\begin{abstract}
We prove under the Bombieri-Lang conjecture for surfaces that there is an absolute bound on the length of sequences of integer squares with constant second differences, for sequences which are not formed by the squares of integers in arithmetic progression. This answers a question proposed in 2010 by J. Browkin and J. Brzezinski, and independently by E. Gonzalez-Jimenez and X. Xarles.

We also show that under the Bombieri-Lang conjecture for surfaces, for every  $k\geq 3$ there is an absolute bound on the length of sequences formed by $k$-th powers with constant second differences. This gives a conditional result on one of Mohanty's conjectures on arithmetic progressions in Mordell's elliptic curves $y^2=x^3+b$. Moreover, we obtain an unconditional result regarding infinite families of such arithmetic progressions. We also study the case of hyperelliptic curves of the form $y^2=x^k+b$.

These results are proved by unconditionally finding all curves of genus $0$ or $1$ on certain surfaces of general type. Moreover, we prove the unconditional analogues of these arithmetic results for function fields by finding all the curves of low genus on these surfaces.
\end{abstract}

\maketitle

\section{Introduction and main results}
In this paper we study certain arithmetic and geometric problems arising from sequences of powers with constant second differences. Let us recall that a sequence $a_1,\ldots, a_n$ has \emph{constant second differences} if for all $3\leq i\leq n,$ one has
$$(a_i-a_{i-1})-(a_{i-1}-a_{i-2})=D,$$
with $D$ not depending on $i$. Equivalently, if for all $4\leq i\leq n$ the following equality holds 
$$a_i-3a_{i-1}+3a_{i-2}-a_{i-3}=0.$$

%Let $k\geq 2$. We are interested in sequences of $k$-th powers with constant second differences.
Any arithmetic progression $x_i=ai+b$ with $1\leq i\leq n$, satisfies that the sequence of its squares has constant second differences equal to $2a^2$; we call such examples \emph{trivial sequences}. It is a classical problem to study the maximal length $M$ of non-trivial sequences of integer or rational squares. %So for $k=2$, the problem is as follows %, so for $k=2$ we actually want to prove that for large enough $n$ there are only trivial sequences of squares with constant second differences.

A sequence of $M$ rational squares with constant second differences corresponds to a quadratic polynomial $f(x)=ax^2+bx+c\in\mathbb{Q}[x]$ such that $f(x)$ is a square for $M$ consecutive integer values of $x$, and a trivial sequence gives a polynomial which is a square. The proof of these facts is elementary, see the Appendix of \cite{breng}. In fact, the same proof applies if we replace $\mathbb{Q}$ by any field of characteristic zero. %Es un ejercicio verificar que toda secuencia de largo \ge d+1 con d-esimas diferencias constantes, es dada por los valores f(1),...,f(j),.. de un polinomio f de grado a lo mas d. Mas aun, la constante comun es k\ne 0 si y solo si $f$ tiene grado exactamente $d$, en cuyo caso el leading coeff de f es k/d!

In 1986 Allison \cite{All} found infinitely many non-trivial integer sequences of length $M=8$, corresponding to polynomials of the form $a(x^2+x)+c$ evaluated at $-3,-2,\ldots, 3,4$ (these are called \emph{symmetric} polynomials). Subsequent work by Bremner \cite{BR} and by Gonzalez-Jimenez and Xarles \cite{XX} gives a complete description of non-trivial sequences of integer squares with constant second differences in the symmetric case, and it turns out that the maximal length in this case is eight, as achieved by Allison. 

In the case where symmetry is not assumed, Browkin and Brzezinski \cite{BrowkinBrzezinski} proved that there are infinitely many non-trivial sequences of six rational squares with constant second differences, and they conjecture that for $M=8$ the only non-trivial rational sequences are the ones found by Allison (and it is known that they cannot be extended to sequences of length nine). %Thus Browkin and Brzezinski conjecture that Problem \ref{sdprobsq} has positive answer with $n=9$. 
However, even the \emph{existence} of some $M$ as required is an open problem:

%This motivates the following problem:

\begin{problem}\label{sdprobsq}
Does there exist an integer $M$ such that the only sequences of rational numbers $x_1,\ldots,x_M$ whose squares have constant second differences correspond to trivial sequences?
\end{problem}

We can also consider the problem for $k$-th powers with a fixed $k\geq 3$. Note that the sequence 1,1,1,...,1 and its scalar multiples $a^k,a^k,...,a^k$ are \emph{degenerate} sequences of $k$-th powers with constant second differences, and one is of course only interested in the non-degenerate case. We propose the following:

\begin{problem}\label{sdprob1} Let $k\geq 3$.
Does there exist an integer $M$ such that there are no sequences of rational numbers $x_1,\cdots,x_M$ whose $k$-th powers form a non-degenerate sequence with constant second differences? %correspond to trivial sequences? 
\end{problem}
Note that in Problem \ref{sdprob1} we expect no analogue of the trivial sequences from Problem \ref{sdprobsq}.

One can relate Problem \ref{sdprob1} to questions about $y$-arithmetic progressions on Mordell elliptic curves. If we have a $y$-arithmetic progression with common differences equal to $d$ on the Mordell elliptic curve $y^2=x^3+b$, that is %$(y+id)^2=x_i^3+b$ for $0\leq i\leq n$,
\begin{eqnarray*}
y_0^2&=&x_0^3+b\cr
(y_0+d)^2&=&x_1^3+b\cr
&\vdots&\cr
(y_0+nd)^2&=&x_n^3+b,
\end{eqnarray*}
then we have 
\begin{eqnarray}\label{lagatita}
2d^2=x_{i+2}^3-2x_{i+1}^3+x_i^3
\end{eqnarray} for all $0\leq i\leq n-2$. Therefore $y$-arithmetic progressions on a Mordell elliptic curve with common differences equal to $d$ give sequences of cubes with constant second differences equal to $2d^2$.

From this we see that Problem \ref{sdprob1} is closely related to the following conjecture of Mohanty \cite{mohanty} on Mordell elliptic curves: \emph{Let $b$ be an integer different from zero. There does not exist length five or greater $y$-arithmetic progression of $\mathbb{Z}$-integral points on $y^2=x^3+b$.}
%%There does not exist length five or greater $x$-arithmetic progression on $y^2=x^3+b$.
%\end{conjecture}

Actually, Mohanty also conjectures an analogous statement for $x$-arithmetic progressions on Mordell elliptic curves, but we will only deal with $y$-arithmetic progressions in this work.
In 1992, Lee and Velez \cite{LV} found infinitely many $y$-arithmetic progressions of length four, five and six (see also \cite{ALV}). Therefore Mohanty's conjecture is not true for the bound five proposed by him, but we still expect that Mohanty's conjecture holds for some larger bound independent of the Mordell curve: 
\begin{conjecture}[Mohanty's conjecture]\label{mohant} There exists an absolute constant $M>0$ such that for any $b\in\mathbb{Q}^*$, there are no length $M$ or greater $y$-arithmetic progression of $\mathbb{Q}$-rational points on $y^2=x^3+b$.
\end{conjecture}
Ulas \cite{ULAS} worked on the analogue of Mohanty's conjecture for hyperelliptic curves of the form $y^2=x^k+b$, in the aspect of $x$-arithmetic progressions.

It is important to note that to consider all $y$-arithmetic progressions on all Mordell curves at the same time, we need to find a bound for the length of sequences of cubes with constant second differences which does not depend on the common differences $d$ of the $y$-arithmetic progression, that is, a bound on sequences of cubes with constant second differences not depending on $2d^2$, see Equation \eqref{lagatita}. % and $b$. %QQQQQIMPORTANTE

Problem \ref{sdprobsq}, Problem \ref{sdprob1} and Conjecture \ref{mohant} remain open and much of the literature around them is focused on constructing examples of long sequences rather than showing the expected bounds on the length. The purpose of this paper is to prove some results in the latter direction, and in fact, we give affirmative answers under the Bombieri-Lang conjecture\footnote{Strictly speaking, we should call it \emph{Bombieri's question}, since we only require it for surfaces. See Conjecture \ref{bombieri} and the comments after it for details.} for surfaces, as well as some unconditional results for the function field analogues.

On the other hand, if instead we consider sequences of powers with \emph{given} second differences, then there are several results in the literature in the case of second differences equal to two. 

The particular case of Problem \ref{sdprob1} when we require the second differences to be equal to two, was studied in the author's thesis \cite{thesis}, and in \cite{subm}. However, such a result is not enough to reach the applications obtained in the present paper; for example, it does not have consequences for Mohanty's conjecture, although it has some other applications of independent interest related to sharpness of bounds, cf.\ \cite{subm}.

The particular case of Problem \ref{sdprobsq} concerning sequences of squares with second differences equal to two is called \emph{B\"uchi's Problem}, which has received special attention due to its applications in Logic related to Hilbert's Tenth Problem (see \cite{lips}, \cite{MAZUR}). Such sequences correspond to \emph{monic} quadratic polynomials, and trivial sequences are formed by the squares of consecutive numbers. Hensley \cite{Hensley} found non-trivial sequences of four integer squares with second differences equal to two, but no example has been found for $M=5$ in the integers. In the case of rational numbers, Lipman (see \cite{MAZUR}) showed that there are infinitely many non-trivial sequences of length $M=5$.

B\"uchi's problem was solved positively in 2000 under the Bombieri-Lang conjecture for surfaces by Vojta \cite{V}, inspired by the work of Bogomolov (cf.\ \cite{Bog}, \cite{B}). Vojta did this by finding all the curves of genus zero or one on the surfaces defined by the equations associated to B\"uchi's problem. % Let us recall the conjecture QQQQNECESARIO?:

%\begin{conjecture}[Bombieri-Lang]
%Let $X$ be a smooth projective surface of general type defined over a number field $K$. We have that $X(K)$ is not Zariski dense in $X$.
%\end{conjecture}

We will generalize and use Vojta's method from \cite{V} to find all curves of genus zero or one on the surfaces defined by the equations associated to Problem \ref{sdprobsq} and Problem \ref{sdprob1}. From this we obtain arithmetic results under the Bombieri-Lang conjecture on the length of nontrivial sequences of $k$-th integer powers with constant second differences, and we prove unconditional analogues for function fields. Finally, we give a conditional result towards Mohanty's conjecture, and an unconditional result on parametric families of $y$-arithmetic progressions on Mordell elliptic curves.

%The geometric method used here to get results towards Problem \ref{sdprobsq} and Problem \ref{sdprob1} originates in work of Vojta \cite{V}, and was spelled out and refined in \cite{subm} (see also Chapter 3 in \cite{thesis}). 
%While the general geometric approach in the present paper is analogous to that of \cite{subm}, the current problem offers its own quirks, leading to different bounds, applications and difficulties.

Our method originates in work of \cite{V}, but one needs a more systematic approach in order to make it work for Problems \ref{sdprobsq} and \ref{sdprob1}. Such a systematic approach was spelled-out and refined in Chapter 3 of the author's thesis, and it was applied in \cite{subm} for sequences of powers with second differences equal to two; we recall this method in Section \ref{explanation}. The problems addressed in this paper, however, lead to some additional complications (for instance, the relevant symmetric differential is more involved) so that when our results are specialized to the context of second differences equal to two, we deduce weaker bounds than in \cite{V} and \cite{subm} (in particular, certain examples of Brody-hyperbolic surfaces of low degree obtained in \cite{subm} do not follow from the present paper). This is not surprising; for instance, at present there are no known sequences of six rational squares with second differences equal to two, while Allison found parametric families of infinitely many sequences of eight integer squares with constant second differences.

Since the geometric part of this paper uses the method discussed in the previous paragraph, some of the routinary computations (such as verifying irreducibility and smoothness of varieties, or computing the genus of some curves) will be similar to the analogous steps in the applications considered in \cite{V} and \cite{subm}. However, the varieties considered in this work are different, and therefore we have included the computations when necessary.

The problem of $k$-th powers with constant second differences (not necessarily equal to two, and not fixing the common value of the second differences) is related to the following projective surfaces
\begin{equation}\label{xnequation} X_{n,k}:\left\{\begin{array}{rl}
x_1^k-3x_2^k+3x_3^k &=x_4^k\cr
&\vdots\cr
x_{n-3}^k-3x_{n-2}^k+3x_{n-1}^k&=x_{n}^k,\end{array}\right.\end{equation}
which in the case $k=2$ were defined by Browkin and Brzezinski \cite{BB}.
For each field $K$ of characteristic zero, we can see that a $K$-rational point on the projective surface $X_{n,k}$ corresponds, up to scaling, to a sequence of length $n$ formed by $k$-th powers in $K$ with constant second differences.

Gonzalez-Jimenez and Xarles \cite{XX} speculated that Vojta's method \cite{V} might be adapted to be applied on these surfaces (with $k=2$), to obtain a conditional solution of Problem \ref{sdprobsq} under the Bombieri-Lang conjecture. We will show in the present paper that this is in fact possible. 
%For each $n$, the surface $X_{n,k}\subset\mathbb{P}^{n-1}$ is defined over $\mathbb{Q}$. By convention $X_{3,k}=\mathbb{P}^2$. 

From now on, by \emph{genus} we mean the \emph{geometric genus} of a curve. Our main result is the following:
%We use Vojta's technique from \cite{V} to prove the following:

\begin{theorem}\label{main}
Let $g\geq 0$ be an integer. 
\begin{itemize}
\item If $k\geq 3$ and $n>\frac{4\max\{g,1\}+1}{k-1}+5$, then there are no curves $C\subseteq X_{n,k}$ of genus $g(C)\leq g$. 
\item If $k=2$ and $n> \max\{10,4g+6\}$, the only curves of genus at most $g$ on $X_{n,k}$ are the lines $[s+t:\pm(2s+t):\cdots :\pm(ns+t)]$ with $[s:t]\in\mathbb{P}^1$.\end{itemize}
\end{theorem}

Note that we do not require the curves to be smooth.

Theorem \ref{main} specializes in the case $g\leq 1$ as follows:

\begin{theorem}\label{main2}
\begin{itemize}
\item If $k\geq 3$ and $n> \frac{5}{k-1}+5$, then there are no curves of genus $0$ or $1$ on $X_{n,k}$. \item If $k=2$ and $n\geq 11$, then the only curves of genus $0$ or $1$ on $X_{n,k}$ are the lines $[s+t:\pm(2s+t):\cdots :\pm(ns+t)]$ with $[s:t]\in\mathbb{P}^1$.\end{itemize}
\end{theorem}

%Hence we are able to find \emph{all} rational or elliptic curves on $X_{n,k}$ for $n\geq \frac{5}{k-1}+5$.
Note that a point $[x_1:\cdots:x_n]$ lying on a curve $[s+t:\pm(2s+t):\cdots:\pm(ns+t)]$ satisfies that $x_i^2=(is+t)^2$, that is, it gives rise to a trivial sequence. %sequence of squares of elements in arithmetic progression.

Using standard techniques of Nevanlinna theory (as in \cite{V}), from Theorem \ref{main2} and the proof of Theorem \ref{main} it is obtained a result regarding Brody-hyperbolicity of these surfaces.

\begin{theorem}\label{hyper}
If $k\geq 3$ and $n> \frac{5}{k-1}+5$, then $X_{n,k}$ is Brody-hyperbolic.
\end{theorem}

Theorem \ref{main} also gives us a result in the arithmetic of function fields.

\begin{theorem}\label{sdffield}
Let $K$ be a function field of genus $g$ with constant field $\mathbb{C}$, and let $n>\frac{4\max\{g,1\}+1}{k-1}+5$. Let $f_1,\ldots,f_n\in K$ be such that the $k$-th powers of this sequence have second differences equal to a fixed $f\in K$. \begin{itemize} \item If $k\geq 3$, then the sequence $(f_1,\ldots, f_n)$ is proportional to a sequence of complex numbers. \item If $k=2$, then the sequence $(f_1,\ldots,f_n)$ is either proportional to a sequence of complex numbers, or it is of the form $f_j=\epsilon_j(aj+b)$ for some $a,b\in K$ and $\epsilon_j\in\left\{-1,1\right\}$.\end{itemize}
\end{theorem}

Browkin and Brzezinski \cite{BB} observed that if one is able to prove a result like Theorem \ref{main2}, then we get the following arithmetic consequence (which we prove in Section \ref{sdproofmain}) which gives a conditional solution to Problem \ref{sdprobsq}:

\begin{theorem}\label{miain}
Assume the Bombieri-Lang conjecture for $\mathbb{Q}$-rational points on the surface $X_{11,2}$. Then there are (up to scaling) finitely many sequences of $11$ rational numbers whose squares have constant second differences, but which are not in arithmetic progression (up to sign). Moreover, there exists an $M>0$ such that if $x_1,\ldots,x_M$ is a sequence of rational squares having constant second differences, then the sequence is trivial.
\end{theorem}

Of course, an analogous result holds for number fields, provided that we assume the Bombieri-Lang conjecture for surfaces over the corresponding number field.

Theorem \ref{miain} formulated in terms of square values of polynomials is as follows:

\begin{theorem}
Assume the Bombieri-Lang conjecture for $\mathbb{Q}$-rational points on the surface $X_{11,2}$. Then, up to scaling by a rational square, there are only finitely many quadratic polynomials $P(t)\in \mathbb{Q}[t]$ which are not the square of a polynomial, and satisfying that the values $P(1),P(2),\ldots,P(11)$ are all squares.
Moreover, there exists an integer $M>0$ such that if $P(t)\in \mathbb{Q}[t]$ is a quadratic polynomial for which the values $P(1),P(2),\ldots,P(M)$ are squares, then $P(t)$ is the square of a polynomial.
\end{theorem}

Observe that in the previous theorem $M$ is independent of the polynomial $P$.

We remark that for B\"uchi's Problem (i.e. for \emph{monic} quadratic polynomials) there is a completely different approach by Pasten \cite{PAS} using a version of the ABC conjecture, and unconditional over function fields. However, despite the strength of his results, Pasten remarks in \cite{PAS}, p.\ 2967 that it is not clear that his method can be adapted to the general case (i.e.\ not necessarily monic polynomials), that we consider here.

%\begin{corollary}
%There are no parametric families of non-trivial sequences of $10$ rational numbers whose squares have constant second differences.
%\end{corollary}

Related to Problem \ref{sdprob1}, we obtain the following result from Theorem \ref{main2}:
\begin{theorem}\label{hey}
Assume the Bombieri-Lang conjecture for $\mathbb{Q}$-rational points on the surfaces $X_{n,k}$ with $k\geq 3$ and $n> \frac{5}{k-1}+5$. Then there are, up to scaling, finitely many sequences of $n$ rational numbers whose $k$-th powers have constant second differences. Moreover, there exists an $M>0$ (depending only on $k$) such that there are no non-degenerate sequences of length $M$ formed by $k$-th rational powers with constant second differences.
\end{theorem}

From this theorem we obtain a result towards Mohanty's conjecture. In order to state this result, let us first introduce some definitions

\begin{definition}
Let $k\geq 3$ and $n\geq 3$ be integers. A $y$\emph{-arithmetic progression} of length $n$ on the curve $y^2=x^k+b$  is a sequence $\left\{(x_j,y_j)\right\}_{j=1}^n$ of points in the curve such that the $y$ coordinates are of the form $y_j=u+vj$ for some $u,v$ with $v\neq 0$.
\end{definition}
We remark that the condition  $v\neq 0$ is imposed as part of the definition in order to avoid uninteresting examples, while the condition $n\geq 3$ is required so that the $y$-coordinates are in a true arithmetic progression. Observe that not all the $x_j$ are equal, because $n\geq 3$ and the $y_j$ are distinct.

\begin{definition}
Let $K$ be a field of characteristic zero. Let $n\geq 3$ be an integer, let $s=\left\{(x_j,y_j)\right\}_{j=1}^n$ and $s'=\left\{(x'_j,y'_j)\right\}_{j=1}^n$ be $K$-rational $y$-arithmetic progressions on the curves $y^2=x^k+b$ and $y^2=x^k+b'$ for some $b,b'\in K^*$ respectively. We say that $s$ and $s'$ are $K$-\emph{equivalent} if there are $\lambda,\mu \in K$ such that $\lambda^k=\mu^2$ and for each $j$ we have $x'_j=\lambda x_j$ and $y'_j=\mu y_j$.
\end{definition}
Note that in this case, $b'=\mu^2 b=\lambda^k b$.

Showing that $\mathbb{Q}$-scalar multiples of a sequence of $k$-th powers correspond to $\mathbb{Q}$-equivalent $y$-arithmetic progressions gives us the following result.
%Note that if we have a $y$-arithmetic progression of length $n$ on the curve $y^2=x^k+b$, then for any $\lambda\in\mathbb{Q}^*$, scaling by $\lambda^{2k}$ gives us a $y$-arithmetic progression of length $n$ on the curve $(\lambda^ky)^2=(\lambda^2 x)^k+\lambda^{2k}b$. When $k$ is even, a similar scaling can be made with $\lambda^k$ instead. When two $y$-arithmetic progressions are related in this way, we will say that they are \emph{equivalent}.

\begin{theorem}\label{moha} Let $k\geq 3$ and $n> \frac{5}{k-1}+5$ . Assume the Bombieri-Lang conjecture for $\mathbb{Q}$-rational points on the surfaces $X_{n,k}$. Then the set of all $\mathbb{Q}$-rational $y$-arithmetic progressions of length $n$ appearing on at least one curve $y^2=x^k+b$ as $b\in \mathbb{Q}^*$ is finite up to $\mathbb{Q}$-equivalence. Moreover, there exists an integer $M_k>0$ (not depending on $b$) such that there are no $y$-arithmetic progressions of length $M_k$ on $y^2=x^k+b$ for any $b\in\mathbb{Q}^*$.
\end{theorem}

We can also obtain the \emph{unconditional} analogous result for function fields of characteristic zero from Theorem \ref{sdffield}. %Here, instead of $\mathbb{Q}$ one considers the function field $K$ of a curve defined over the complex numbers, and the $y$-arithmetic progressions are required to satisfy the necessary condition that the sequence of $y$-coordinates is not proportional to a sequence of complex numbers.
Here, instead of $\mathbb{Q}$ one considers the function field $K$ of a curve defined over the complex numbers. Note that over $\mathbb{C}$ we can have arbitrarily large $y$-arithmetic progressions on curves $y^2=x^k+b$ with $b$ complex. The next result shows that for function fields, these sequences of complex points are the only long sequences up to equivalence.

\begin{theorem}\label{catorce} 
Let $C$ be a smooth projective curve over $\mathbb{C}$ of genus $g$, and let $K$ be the function field of $C$. Let $k\geq 3$  and  $n> \frac{4\max\{g,1\}+1}{k-1}+5$ be positive integers. Every $K$-rational $y$-arithmetic progression of length $n$ on a curve of type $y^2=x^k+B$ with $B\in K^*$, is $K$-equivalent to a $\mathbb{C}$-rational $y$-arithmetic progression on a curve of type $y^2=x^k+b$ with $b\in \mathbb{C}^*$. 
\end{theorem}

%Similarly, we obtain a result for hyperelliptic curves of the form $y^2=x^k+b$.
%\begin{corollary} Let $k>4$, and let $C:y^2=x^k+b$ be an hyperelliptic curve. Assume the Bombieri-Lang conjecture for the surfaces $X_{n,k}$ with $n\geq 7$. Then there are finitely many $y$-arithmetic progressions of length $7$ on $C$. There exists an $M>0$ (not depending on $b$) such that there are no $y$-arithmetic progressions of length $M$ on $C$.
%\end{corollary}
%The unconditional analogue for function fields is as follows:
%\begin{corollary}
%Let $C$ be a curve of genus $g$, and let $K(C)$ be the function field associated to $C$. Let $E:y^2=x^k+b$ be an hyperelliptic curve on the function field, with $b\in K(C)$ non-constant. There are finitely NOQQQmany $y$-arithmetic progressions of length $\lceil\frac{4g+1}{k-1}+6\rceil$ on $E$.% There exists an $M_k>0$ (not depending on $b$) such that there are no $y$-arithmetic progressions of length $M_k$ on $E$.
%\end{corollary}

We can also obtain unconditional results on parametric families of sequences for Mohanty's conjecture. To state the result in a precise way, we need the following definitions.

\begin{definition} 
Let $k\geq 3$ be an integer. A \emph{parametric family} of $\mathbb{Q}$-rational $y$-arithmetic progressions of length $n\geq 3$ on curves of type $y^2=x^k+b$ is a $K$-rational  $y$-arithmetic progression $\left\{(f_j,g_j)\right\}_{j=1}^n$ on $y^2=x^3+B$ for some $B\in K^*$, where $K$ is the function field of a smooth projective curve $C$ defined over $\mathbb{Q}$ such that $C(\mathbb{Q})$ is infinite. % (thus, $C$ is isomorphic to either $\mathbb{P}^1_\mathbb{Q}$ or an elliptic curve of non-zero rank by Faltings theorem).
\end{definition}

The idea behind this definition is the following observation: except for the finitely many $a\in C(\mathbb{Q})$ that are poles of some $f_j$ or $g_j$, or that are zeros of $B$, for each rational point $a\in C(\mathbb{Q})$ we obtain a $\mathbb{Q}$-rational $y$-arithmetic progression $\left\{(f_j(a),g_j(a))\right\}_{j=1}^n$ on the curve $y^2=x^k+B(a)$ (which is defined over $\mathbb{Q}$). The rational points $a\in C(\mathbb{Q})$ are the parameters for the family.  This definition is consistent with the existing literature on parametric families of examples for Mohanty's conjecture, see for instance \cite{LV}.

Given a $\mathbb{Q}$-rational $y$-arithmetic progression $s$ on some curve $y^2=x^k+b$, the scaling process can be used to produce a parametric family starting  from $s$. These parametric families do not provide a genuine source of infinitely many $y$-arithmetic progressions, and it is natural to distinguish them from the other parametric families.

\begin{definition} A parametric family $\mathcal{F}=\left\{(f_j,g_j)\right\}_{j=1}^n$ of $\mathbb{Q}$-rational arithmetic progression  (for the function field $K$ of a curve $C$ defined over $\mathbb{Q}$) is \emph{non-genuine} if there are $\phi,\psi \in K$ and a $\mathbb{Q}$-rational $y$-arithmetic progression $\left\{(x_j,y_j)\right\}_{j=1}^n$ on a curve $y^2=x^k+b$ with $b\in\mathbb{Q}^*$, such that $\psi^k=\phi^2$ and for each $j$ we have $f_j=\psi x_j$ and $g_j=\phi y_j$. Otherwise, we say that $\mathcal{F}$ is \emph{genuine}. % parametric family.
\end{definition}

From Theorem \ref{sdffield} we obtain the following unconditional result regarding parametric families of $y$-arithmetic progressions.

\begin{theorem}\label{diecisiete}
Let $k\geq 3$ be an integer and let $n> \frac{5}{k-1} +5$. There are no genuine parametric families of $\mathbb{Q}$-rational $y$-arithmetic progressions of length $n$ on curves of type $y^2=x^k+b$.
\end{theorem}

%In the case $k=3$, from \cite{LV} and \cite{ALV} we know that there are genuine parametric families (parametrized by elliptic curves)  of $\mathbb{Q}$-rational $y$-arithmetic progressions of length $6$ on Mordell's elliptic curves. The previous theorem shows that there are no genuine parametric families of length $8$. The question about the existence of such families for length $7$ remains open.

%From Theorem \ref{main2} we obtain an unconditional result regarding parametric families of $y$-arithmetic progressions on curves $y^2=x^k+b$. 
%
%\begin{corollary} Let $\mathcal{S}_{n,k}$ be the set of all $y$-arithmetic progressions of length $n$ (up to equivalence) that appear on at least one curve $y^2=x^k+b$, with $b\in\mathbb{Q}$. There are no parametric families of elements in $\mathcal{S}_{\lceil\frac{5}{k-1}+5\rceil,k}$.%,%of length $8$ on the curve $y^2=x^3+b$. 
%%Let $k\geq 4$. There are no parametric families of $y$-arithmetic progressions of length $\lceil\frac{5}{k-1}+5\rceil$ on the curve $y^2=x^k+b$.
%\end{corollary}

%For hyperelliptic curves we have the following
%\begin{corollary}
%Let $C$ be a curve of genus $g$, and let $K(C)$ be the function field associated to $C$. Let $E:y^2=x^k+b$ be a hyperelliptic curve on the function field, with $b\in K(C)$ non-constant. There are finitely many $y$-arithmetic progressions of length $\frac{4g+1}{3}+6$ on $E$. There exists an $M>0$ (not depending on $b$) such that there are no $y$-arithmetic progressions of length $M$ on $E$.
%\end{corollary}

Part of this work (the case $k=2$) corresponds to Chapter 4 of the author's PhD thesis at Queen's University. I thank my supervisor Professor Ernst Kani for his dedicated supervision and helpful revisions on early versions of this work. I thank Manfred Kolster for his feedback and suggestions that led me to extend the work from my thesis to higher powers. I also thank Andrew Bremner, Juliusz Brzezinski and Xavier Xarles for their comments on previous versions of this work and for pointing out some useful references. %I thank Xavier Xarles for commenting on previous versions of this work.

%%%%% DE AQUI EMPECE POTENCIAS KESIMAS

\section{Outline of the method}\label{explanation}

In this section we recall the method given in Chapter 3 of \cite{thesis}, which is an extension of the method implicit in Vojta's work \cite{V}. As in \cite{subm}, we include some improvements to \cite{V} which allow us to obtain better bounds in applications.
%This is an account of the method implicit in Vojta's work \cite{V}, inspired by work of Bogomolov. For analytic proofs of these facts in the particular case studied by Vojta, see \cite{V}. For the general case treated in an algebraic setting, see \cite{thesis}. In particular Theorem 3.87 in \cite{thesis} is an improvement of Lemma 2.10 in \cite{V}, and permits us to obtain better numerical results.
%We outline this method in some generality beyond what we need here because it can be useful in other applications.

Let us first list the relevant definitions
%We work with the notion of $\omega$-integral curve from \cite{V}:

\begin{definition}\label{omegaint}
Let $X$ be a smooth variety over a field of characteristic zero, let $\mathcal{L}$ be an invertible sheaf on $X$ and let $\omega\in H^0(X,\mathcal{L}\otimes S^r\Omega^1_{X/\mathbb{C}})$, where $r$ is an integer. An irreducible curve $C$ on $X$ is said to be $\omega$\emph{-integral} if the image of the section $\varphi^*_C\omega$ in $H^0(\tilde{C},\varphi^*_C\mathcal{L}\otimes S^r\Omega^1_{\tilde{C}})$ is zero, where $\varphi_C:\tilde{C}\to X$ is the normalization of $C\subset X$. 
\end{definition}

Let $X$ be a smooth irreducible surface. Let $U\subset X$ be an affine open subset such that $\mathcal{L}_{|U}\cong \mathcal{O}_{X|U}$ and so that there are $u,v\in \mathcal{O}_X(U)$ satisfying that $du,dv$ form a basis of $\Omega^1_{X/\mathbb{C}}(U)$. In this open subset, any given $\omega\in H^0(X,\mathcal{L}\otimes S^r\Omega^1_{X/\mathbb{C}})$ has the form $\sum_{i=0}^r A_i(du)^{r-i}(dv)^i$, where $A_i\in \mathcal{O}_U(U)$. We can then define the following:

\begin{definition}\label{delta}
Let $\delta\in K:=k(X)$ be the discriminant of the monic polynomial $\Sigma_{i=0}^r\frac{A_i}{A_0}T^{r-i}\in K[T]$. We define the \emph{discriminant} of $\omega$ to be the Zariski closed set $$\Delta_U:=(X/U)\cup \mathbb{V}_U(A_0)\cup \mathbb{V}_{U\backslash \mathbb{V}_U(A_0)}(\delta)\subset X$$ (where $\mathbb{V}_U(A_0)$ denotes the zero locus of $A_0\in \mathcal{O}_U(U)$ on $U$).
\end{definition}

\begin{definition}
Let $f:Y\to Z$ be a morphism of varieties, and let $\mathcal{L}$ be an invertible sheaf on $Z$. We define $$f^\bullet: H^0(Z,\mathcal{L}\otimes S^r\Omega^1_{Z/\mathbb{C}})\to H^0(Y,f^*\mathcal{L}\otimes S^r\Omega^1_{Y/\mathbb{C}})$$ to be the map induced by the following composition
\begin{eqnarray*}
\mathcal{L}\otimes S^r\Omega^1_{Z/\mathbb{C}} &\to & f_*f^*(\mathcal{L}\otimes S^r\Omega^1_{Z/\mathbb{C}})\cr
&\to &f_*(f^*\mathcal{L}\otimes f^*S^r\Omega^1_{Z/\mathbb{C}})\cr
&\to &f_*(f^*\mathcal{L}\otimes S^rf^*\Omega^1_{Z/\mathbb{C}})\cr
&\to &f_*(f^*\mathcal{L}\otimes S^r\Omega^1_{Y/\mathbb{C}}).
\end{eqnarray*}
\end{definition}

\begin{definition}
Given a smooth irreducible surface $X$, an effective Cartier divisor $D$ of $X$ with associated subscheme $Y=Y_D$, a locally free sheaf $\mathcal{F}$ and a section $s\in H^0(X,\mathcal{F})$, we say that $s$ \emph{vanishes identically along} $D$ if the image of $s$ under the map $H^0(X,\mathcal{F})\to H^0(X,\mathcal{F}\otimes i_{Y*}\mathcal{O}_Y)$ is zero.
\end{definition}

The general outline of the method is as follows: Given a smooth projective surface $X$ embedded in a fixed projective space, we can attempt to find the curves on $X$ with geometric genus bounded by a number $g$ as follows:
\begin{itemize}
\item We choose a morphism $\pi:X\to\mathbb{P}^2$, an invertible sheaf $\mathcal{L}$ on $\mathbb{P}^2$, and a section $\omega\in H^0(\mathbb{P}^2,\mathcal{L}\otimes S^r\Omega^1_{\mathbb{P}^2/\mathbb{C}})$ so that each irreducible component of the branch divisor $B$ of $\pi$ is $\omega$-integral.
\item Find all $\omega$-integral curves in $\mathbb{P}^2$.
\item Find all $\pi^\bullet\omega$-integral curves in $X$.
\item If $\mathcal{L}$ satisfies that $\deg_{\tilde{C}}(\varphi^*_C\mathcal{L})<r(2-2g)$ for all curves $C$ of genus less than $g$, then all curves of genus less than $g$ will be $\pi^\bullet\omega$-integral. %We look in our list of $\pi^\bullet\omega$-integral the ones with genus less than $g$.
\item If the previous condition is not satisfied, one modifies $\mathcal{L}$ and $\omega$ using the fact that the branch curves of $\pi$ are $\omega$-integral.
\end{itemize}
So finally, to find all curves of genus less than $g$ on $X$, we only have to check which curves in our list of $\pi^\bullet\omega$-integral curves have genus less than $g$.

Now we explain in detail the method. To construct the differential $\omega$ we heavily use the condition that the irreducible components of the branch divisor are $\omega$-integral, which will be useful in further steps of the method. Vojta instead uses computer search in positive characteristic. In Appendix A of \cite{V}, he explains how he found his differential for B\"uchi's Problem.% the application considered in \cite{V}, using computer search in positive characteristic.

Then we need a way to find $\omega$-integral curves. The following result will allow us to check for any $C\subset X$ whether or not $C$ is $\omega$-integral (see Corollary 3.72 in \cite{thesis}):
\begin{theorem}\label{TYU}
Let $X$ be a smooth surface, let $\omega\in H^0(X,\mathcal{L}\otimes S^r\Omega^1_{X/\mathbb{C}})$. Let $C$ be an irreducible curve in $X$, let $U=\mathrm{Spec}(A)$ be an open set in $X$ such that $C\cap V\neq \emptyset$, and $\mathcal{L}_{|U}\cong \mathcal{O}_U$. Let $I=(g)$ be an ideal in $A$ such that $C\cap U$ is generated by $I$. Let $\omega_0$ be the image of $\omega$ by the maps $$H^0(X,\mathcal{L}\otimes S^r\Omega^1_{X/\mathbb{C}})\to H^0(U,\mathcal{L}\otimes S^r\Omega^1_{X/\mathbb{C}})\stackrel{h}{\rightarrow} H^0(U,S^r\Omega^1_{U/\mathbb{C}})=S^r\Omega^1_{A/\mathbb{C}}.$$ If $\omega_0\in S^r\Omega^1_{A/\mathbb{C}}$ lies in $gS^r(\Omega^1_{A/\mathbb{C}})+dgS^{r-1}\Omega^1_{A/\mathbb{C}}$, then $C$ is $\omega$-integral.
\end{theorem}

%Once we find a list of $\omega$-integral curves on $X$, we want to check if this list consists of \emph{all} $\omega$-integral curves of $X$. One can do this by defining the \emph{discriminant} of $\omega$, which permits us to locally count the number of $\omega$-integral curves passing through a point. Fix a non-empty open subset $V$ on $X$ such that there are regular functions $u,v\in\mathcal{O}_X(V)$ with the property that $du,dv$ are a basis of $\Omega^1_{X/\mathbb{C}}(V)$ as $\mathcal{O}_X(V)$-module.
%Let $U\subset V$ be a non-empty basic affine open set such that $\mathcal{L}_{|U}\cong\mathcal{O}_U$. Under the isomorphism $H^0(U,\mathcal{L}\otimes S^r\Omega^1_{X/\mathbb{C}})\cong H^0(U,S^r\Omega^1_{X/\mathbb{C}})$ we have that the image of $\omega_{|U}$ in $H^0(U,S^r\Omega^1_{X/\mathbb{C}})$ can be written as $\Sigma_{i=0}^r A_i(du)^{r-i}(dv)^i$ with $A_i\in \mathcal{O}_U(U)$.

Once we find some $\omega$-integral curves in $\mathbb{P}^2$, we want to check that we have all of them. For this we count the number of $\omega$-integral curves passing through each point on a Zariski open subset of $\mathbb{P}^2$ (which we can determine from the equations of $\omega$). %One can count the number of $\omega$-integral curves passing through any point $P$ outside of $\Delta_U$ by the following result (see Theorem 3.76 in \cite{thesis}): 
The following result allows us to do this (see Theorem 3.76 in \cite{thesis}):
\begin{theorem}\label{eee} Let $\omega\in H^0(X,\mathcal{L}\otimes S^r\Omega^1_{X/\mathbb{C}})$, and let $\Delta_U$ be defined as before. For any given point $P\in X\backslash \Delta_U$ there are at most $r$ $\omega$-integral curves passing through $P$. More precisely, the sum of the multiplicities $\mu_P(C)$ for all $\omega$-integral curves $C$ passing through $P$ is at most $r$.
\end{theorem}
Hence, if we find $r$ $\omega$-integral curves for a point $P\in X\backslash \Delta_U$ we know that these are all the $\omega$-integral curves passing through $P$. Verifying that there are $r$ $\omega$-integral curves passing through each point of $X\backslash U$ and checking if the component curves of $\Delta_U$ are $\omega$-integral (using Theorem \ref{TYU}) leads us to know that our list of $\omega$-integral curves of $X$ is complete.

%Usually, one defines $\omega\in H^0(X,\mathcal{L}\otimes S^r\Omega^1_{X/\mathbb{C}})$  and applies the previous theorem on a surface which is easier to understand (for example $X=\mathbb{P}^2$ or $X=\mathbb{P}^1\times \mathbb{P}^1$). Here we work with $X_{2,k}=\mathbb{P}^2$. 
%
%Once we find all $\omega$-integral curves in $X_{2,k}=\mathbb{P}^2$, we can use the morphism $\rho_{n,k}:X_{n,k}\to X_{2,k}$ to find all $\omega_n$-integral curves on $X_{n,k}$, for $\omega_n$ a suitable differential in $X_{n,k}$ which depends on $\omega$. Here we consider $\omega_n=\rho_{n,k}^\bullet\omega$, where $\rho^\bullet_{n,k}$ is the induced map on global sections from the $\mathcal{O}_{X_{2,k}}$-homomorphism
%\begin{eqnarray*}
%\mathcal{L}\otimes S^2\Omega^1_{X_{2,k}/\mathbb{C}} &\to & \rho_{n,k*}\rho_{n,k}^*(\mathcal{L}\otimes S^2\Omega^1_{X_{2,k}/\mathbb{C}})\cr
%&\to &\rho_{n,k*}(\rho_{n,k}^*\mathcal{L}\otimes\rho_{n,k}^*S^2\Omega^1_{X_{2,k}/\mathbb{C}})\cr
%&\to &\rho_{n,k*}(\rho_{n,k}^*\mathcal{L}\otimes S^2\rho_{n,k}^*\Omega^1_{X_{2,k}/\mathbb{C}})\cr
%&\to &\rho_{n,k*}(\rho_{n,k}^*\mathcal{L}\otimes S^2\Omega^1_{X_{n,k}/\mathbb{C}}),
%\end{eqnarray*}
%where the last map is induced by the morphism $$(\rho_{n,k})_{X_{n,k}/X_{2,k}/\mathbb{C}}:\rho^*_{n,k}\Omega^1_{X_{2,k}/\mathbb{C}}\to \Omega^1_{X_{n,k}/\mathbb{C}}$$ from \cite{EGAEGA} IV.16.4.19.1. (Note that translating Definition \ref{omegaint} into this language, we say that a curve $C$ on a surface $X$ is $\omega$-integral if and only if $\varphi_C^\bullet\omega=0$.)

After we find all $\omega$-integral curves in $\mathbb{P}^2$, we want to find all $\pi^\bullet\omega$-integral curves in $X$. They will be exactly the pullbacks under $\pi$ of the $\omega$-integral curves thanks to the following result (see Theorem 3.35 in \cite{thesis})

%The following result will allow us to find all $\omega_n$-integral curves on $X_{n,k}$ (see Theorem 3.35 in \cite{thesis})

\begin{theorem}\label{bullet}
Let $\pi:X'\to X$ be a morphism of smooth surfaces. Let $C'\subset X'$ be an irreducible curve and $C=\pi(C')$ be an irreducible curve on $X$. Let $\mathcal{L}$ be an invertible sheaf on $X$ and let $\omega\in H^0(X,\mathcal{L}\otimes S^r\Omega^1_{X/\mathbb{C}})$. The following are equivalent
\begin{itemize}
\item The curve $C$ is $\omega$-integral;
\item The curve $C'$ is $\pi^\bullet\omega$-integral.
\end{itemize}
\end{theorem}

%Hence we can easily compute the set of all $\omega_n$-integral curves on $X_{n,k}$ once we have found all $\omega$-integral curves on $X_{2,k}=\mathbb{P}^2$.

The last step is to prove that any irreducible curve $C\subset X_{n,k}$ of genus less than or equal to $g$ is in the list of $\omega_n$-integral curves, for $n$ sufficiently large (depending on $g$). This is done by showing that the degree of the sheaf  $\varphi_C^*\mathcal{L}\otimes S^r\Omega^1_{\tilde{C}/\mathbb{C}}$ (from Definition \ref{omegaint}) over the normalization $\tilde{C}$ of a curve of genus less than or equal to $g$ on $X_{n,k}$ is negative, so the section $\varphi^\bullet_C\omega_n$ is forced to be zero.

If a sheaf $\varphi_C^*\mathcal{L}$ does not satisfy that condition, we can still find another invertible sheaf $\mathcal{L}'$ on $X$ for which it is more likely to satisfy the condition. From the following result (see Theorem 3.87 in \cite{thesis} and Lemma 2.10 in \cite{V}) we see how to use the fact that all irreducible components of the branch divisor of $\pi$ are $\omega$-integral.

%If this degree is not negative for the sheaf $\mathcal{L}$ we chose, we can find a better one when the branch curves of the morphism $\rho_{n,k}$ are $\omega$-integral. This requires the following definition:

%The following result (see Theorem 3.87 in \cite{thesis}), is a generalization of Lemma 2.10 in \cite{V} to morphisms with higher ramification.

\begin{theorem}\label{elteo1}
Let $X$ and $X'$ be smooth integral surfaces defined over $\mathbb{C}$. Let $\pi:X'\to X$ be a dominant morphism and let $D\subseteq X'$ be a prime divisor such that $C=\pi(D)$ is a curve. Suppose that $\pi$ has ramification index $e=e_{D/C}(\pi)>1$ at $D$. Let $\mathcal{L}$ be an invertible sheaf on $X$, let $r$ be a positive integer, and let $\omega\in H^0(X,\mathcal{L}\otimes S^r\Omega^1_{X/\mathbb{C}})$. If $C$ is $\omega$-integral, then $\pi^\bullet\omega\in H^0(X',\pi^*\mathcal{L}\otimes S^r\Omega^1_{X'/\mathbb{C}})$ vanishes identically along $(e-1)D$.
\end{theorem}

Note that in Vojta's work the conclusion of Lemma 2.10 is that $\pi^\bullet\omega$ vanishes identically along $D$. It was unnecessary in his work to consider higher ramification indices. However, in this work it will be very useful to have $(e-1)D$ instead of $D$ to obtain better bounds in applications.
Since $\pi^\bullet\omega$ vanishes identically along $(e-1)D$, there exists a unique section $\omega'\in H^0(X,\mathcal{O}(-(e-1)D)\otimes\pi^*\mathcal{L}\otimes S^r\Omega^1_{X/\mathbb{C}})$ and such that all $\omega'$-integral curves are $\pi^\bullet\omega$-integral. See Proposition 3.88 in \cite{thesis} or Corollary 2.11 in \cite{V} for details.

%With $\mathcal{F}=\mathcal{L}\otimes S^r\Omega^1_{X/\mathbb{C}}$, and by using the exact sequence $$0\rightarrow \mathcal{O}_X(-D)\rightarrow \mathcal{O}_X\rightarrow i_{Y*}\mathcal{O}_Y\rightarrow 0 $$ on \cite{MLN}, p.\ 63, we obtain the following result, which combined with Theorem \ref{elteo1} allows us to prove that the invertible sheaf $\mathcal{L}'=\mathcal{O}(-(e-1)D)\otimes \pi^*\mathcal{L}$, which behaves better than $\mathcal{L}$, has a section $\pi^\bullet\omega$ making all $\omega'$-integral curves to be $\omega$-integral (see Proposition 3.88 in \cite{thesis} or Corollary 2.11 in \cite{V}).0
 %and a section $\omega'\in H^0(X,\mathcal{L}'\otimes\Omega^1_{X/\mathbb{C}})$ such that all $\omega'$-integral curves are $\omega$-integral 
 
\begin{proposition}\label{laprop2}
Let $X$ be a smooth integral surface and let $\omega\in H^0(X,\mathcal{L}\otimes S^r\Omega^1_{X/\mathbb{C}})$. Let $D$ be an effective divisor on $X$. Suppose that $\omega$ vanishes identically along $D$. Then there is a symmetric differential $\omega'\in H^0(X,\mathcal{O}(-D)\otimes \mathcal{L}\otimes S^r\Omega^1_{X/\mathbb{C}})$ such that all $\omega'$-integral curves are among the $\omega$-integral curves.
\end{proposition}

\section{Preliminary study of the surfaces $X_{n,k}$}

In this section, we will study the varieties $X_{n,k}\subseteq \mathbb{P}^{n-1}$ given by Equations \eqref{xnequation}, and we will describe some useful properties about them. By convention, for any $k\geq 2$ we let $X_{3,k}=\mathbb{P}^2$.

%\begin{lemma}\label{sdxm}
%If $[x_1:\ldots:x_n]\in X_{n,k}$, then we have
%$$\frac{(m-3)(m-2)}{2}x_1^k-((m-2)^2-1)x_2^k+\frac{(m-2)(m-1)}{2}x_3^k = x_m^k,$$
%for any $1\leq m\leq n$.
%\end{lemma} 
%
%\begin{proof}
%Fix $n\geq 4$. We will prove this lemma by induction on $m$. It can be seen that the lemma is true for $m=1,2,3$. We know that in $X_{n,k}$ we have $x_1^k-3x_2^k+3x_3^k=x_4^k$, so the lemma is also true for $m=4$. Now let $4<m\leq n$ and suppose that the statement is true for all $i<m$. From the equation $x_{m-3}^k-3x_{m-2}^k+3x_{m-1}^k=x_m^k$ of \eqref{xnequation} we obtain by the induction hypothesis that
%\begin{eqnarray*}
%x_m^k &=& x_{m-3}^k-3x_{m-2}^k+3x_{m-1}^k\cr
%&=& \left(\frac{(m-6)(m-5)}{2}x_1^k-((m-5)^2-1)x_2^k+\frac{(m-5)(m-4)}{2}x_3^k\right)\cr 
%&&-3\left(\frac{(m-5)(m-4)}{2}x_1^k-((m-4)^2-1)x_2^k+\frac{(m-4)(m-3)}{2}x_3^k\right)\cr 
%&&+3\left(\frac{(m-4)(m-3)}{2}x_1^k-((m-3)^2-1)x_2^k+\frac{(m-3)(m-2)}{2}x_3^k\right)\cr
%&=& \frac{(m-3)(m-2)}{2}x_1^k-((m-2)^2-1)x_2^k+\frac{(m-2)(m-1)}{2}x_3^k.
%\end{eqnarray*}
%\end{proof}

The following result can be proved by induction.

\begin{lemma}\label{tusfi}
Let $f_i= x^k_i-3x^k_{i+1}+3x^k_{i+2}-x^k_{i+3}$ be the generators of the ideal defining $X_{n,k}$, and let
$$
g_i=\frac{i(i+1)}{2}x_1^k-((i+1)^2-1)x_2^k + \frac{(i+1)(i+2)}{2}x_3^k-x_{k+3}^k.
$$
Then we have the equality of ideals $(f_1,\ldots,f_{n-3})=(g_1,\ldots,g_{n-3})$ in $\mathbb{C}[x_1,\ldots,x_n]$. In particular, the $g_i$ for $1\le i\le n-3$ are also defining equations for $X_{n,k}$, and for any $1\leq m\leq n$ the points $[x_1:\ldots:x_n]\in X_{n,k}$ satisfy $$\frac{(m-3)(m-2)}{2}x_1^k-((m-2)^2-1)x_2^k+\frac{(m-2)(m-1)}{2}x_3^k = x_m^k.$$\end{lemma}

The following observation will be useful for proving several smoothness results.

\begin{lemma}\label{sdnozero}
If $[x_1:\cdots:x_n]$ is a point on $X_{n,k}$, then no three of $x_1,\ldots,x_n$ can be simultaneously zero.
\end{lemma}

\begin{proof}
This follows from the last statement in Lemma \ref{tusfi}. %If $x_1,x_2,x_3$ are all zero, then by Lemma \ref{tusfi} every $x_i$ is equal to zero, which contradicts the fact that $[x_1:\cdots:x_n]\in\mathbb{P}^{n-1}$. 
%Now view $$\frac{(j-3)(j-2)}{2}x_1^k-((j-2)^2-1)x_2^k+\frac{(j-1)(j-2)}{2}x_3^k=0$$ as an equation in $j$. Note that for $j=1$ this gives $x_1^k=0$, for $j=2$ this gives $x_2^k=0$ and for $j=3$ it gives $x_3^k=0$. This equation can be written in the form 
%\begin{eqnarray}\label{aaaa}
%(x_1^k-2x_2^k+x_3^k)j^2-(-5x_1^k+8x_2^k-3x_3^k)j+(6x_1^k-6x_2^k+2x_3^k)&=&0.
%\end{eqnarray}
%If all $x_1^k-2x_2^k+x_3^k$, $-5x_1^k+8x_2^k-3x_3^k$ and $6x_1^k-6x_2^k+2x_3^k$ are zero, then it can be computed that $x_1=x_2=x_3=0$, which is not possible. Hence Equation \eqref{aaaa} has at most two solutions in $j$. Therefore there are at most two values of $j$ such that $x_j=0$.
\end{proof}

We now will study some morphisms between the surfaces $X_{n,k}$.
For $n\geq 4$, let 
%\begin{eqnarray*}
%\tilde{\pi}_n:\mathbb{P}^{n-1}\backslash\left\{[0:\cdots:0:1]\right\}&\to&\mathbb{P}^{n-2}\cr
%[x_1:\cdots :x_{n}]&\mapsto&[x_1:\cdots:x_{n-1}].
%\end{eqnarray*}
%be the rational map corresponding to the inclusion $\mathbb{C}[x_1,\ldots,x_{n-1}]\to \mathbb{C}[x_1,\ldots,x_n]$. We then let 
$$\pi_n:X_{n,k}\to X_{n-1,k}$$ be the restriction of the linear projection 
\begin{eqnarray*}
\tilde{\pi}_n:\mathbb{P}^{n-1}\backslash\{[0:\ldots:0:1]\} &\to& \mathbb{P}^{n-2}\cr
[x_1:\ldots:x_{n}] &\mapsto & [x_1:\ldots:x_{n-1}]
\end{eqnarray*}
%$\tilde{\pi}_n$ to $X_{n,k}$, that is, the map corresponding to 
%$$\mathbb{C}[x_1,\ldots,x_{n-1}]/(f_1,\ldots,f_{n-4})\to \mathbb{C}[x_1,\ldots,x_n]/(f_1,\ldots,f_{n-3}),$$ 
%which exists because $(f_1,\ldots,f_{n-4})\subseteq \mathbb{C}[x_1,\ldots,x_{n-1}]\cap (f_1,\ldots,f_{n-3})$. %Therefore $\pi_n(X_{n,k})\subseteq X_{n-1,k}$. 
%For each $n\geq 4$, define the map $\pi_n:X_{n,k}\to X_{n-1,k}$ as the restriction to $X_{n,k}$ of the morphism
%\begin{eqnarray*}
%\tilde{\pi}_n:\mathbb{P}^{n-1}\backslash\left\{[0:\cdots:0:1]\right\}&\to&\mathbb{P}^{n-2}\cr
%[x_1:\cdots :x_{n}]&\mapsto&[x_1:\cdots:x_{n-1}].
%\end{eqnarray*}
%The rational map $\tilde{\pi}_n$ corresponds to the inclusion map $\mathbb{C}[x_1,\ldots,x_{n-1}]\to \mathbb{C}[x_1,\ldots,x_n]$ (which respects the grading) in the sense of \cite{H} II, Exercise 2.14(b), and the morphism $\pi_n$ corresponds to the induced map $$\mathbb{C}[x_1,\ldots,x_{n-1}]/(f_1,\ldots,f_{n-4})\to \mathbb{C}[x_1,\ldots,x_n]/(f_1,\ldots,f_{n-3}),$$ which exists because $(f_1,\ldots,f_{n-4})\subseteq \mathbb{C}[x_1,\ldots,x_{n-1}]\cap (f_1,\ldots,f_{n-3})$. Therefore $\pi_n(X_{n,k})\subseteq X_{n-1,k}$. 

\begin{proposition}\label{rir}
For each $n$, the map $\pi_n$ is a finite surjective morphism of degree $k$, ramified only at the components of $R_n=\mathrm{div}_{X_{n,k}}(x_n)$. %with ramification curve $R_n:=\left\{x_n=0\right\}\cap X_{n,k}$.
\end{proposition}

\begin{proof}
If $\pi_n([x_1:\cdots:x_n])$ is undefined, then we get $x_1=\cdots =x_{n-1}=0$. By Lemma \ref{tusfi} we obtain that $x_n=0$, contradicting the fact that $[x_1:\cdots :x_n]\in\mathbb{P}^{n-1}$. Therefore $\pi_n$ is a morphism.

Now let $P=[x_1:\cdots:x_{n-1}]\in X_{n-1}$, and let $\tilde{P}:=[x_1:\ldots:x_n]\in\mathbb{P}^{n-1}$ be a preimage of $P$ with respect to $\tilde{\pi}_n$. Then $\tilde{P}$ lies on $X_{n,k}$ if and only if $$x_n^k=\frac{(n-2)(n-1)}{2}x_3^k-((n-2)^2-1)x_2^k+\frac{(n-3)(n-2)}{2}x_1^k,$$
by Lemma \ref{tusfi}. Since we can always solve this, we have that $\pi_n$ is a surjective quasi-finite morphism. Moreover, it is finite of degree $k$ because it is projective. 
%It is of degree $k$ by Lemma \ref{ordg}.QQQQ

The curve $$\frac{(n-2)(n-1)}{2}x_3^k-((n-2)^2-1)x_2^k+\frac{(n-3)(n-2)}{2}x_1^k=0$$ in $\mathbb{P}^2$ is irreducible when $n\neq 1,2,3$. Since each point on this curve has only one preimage under $\tilde{\pi}_n$ in $X_{n,k}$, it follows from Lemma \ref{tusfi} that $\pi_n$ is totally ramified at each component of $R_n$. %From Lemma \ref{sdxm} we have that the pullback of this curve on $X_{n,k}$ has equation $x^k_n=0$.
\end{proof}

We have a tower of finite morphisms:
$$\mathbb{P}^2= X_{3,k}\stackrel{\pi_4}{\leftarrow}X_{4,k}\stackrel{\pi_5}{\leftarrow}X_{5,k}\stackrel{\pi_6}\leftarrow\cdots$$
Define $\rho_{n,k}=\pi_4\circ\cdots\circ\pi_n$. This morphism is finite of degree $k^{n-3}$.
Note that by Lemma \ref{tusfi}, the image under $\rho_{n,k}$ of $R_n$ in the surface $X_{3,k}=\mathbb{P}^2$ is $$C_n:\frac{(n-3)(n-2)}{2}x_1^k-((n-2)^2-1)x_2^k+\frac{(n-2)(n-1)}{2}x_3^k=0,$$
so $\rho_{n,k}(R_n)=C_n$.
%Denote by $U_i=\left\{[x_1:\ldots:x_n]\in \mathbb{P}^{n-1}:x_i\neq 0\right\}$ the affine open covering of $\mathbb{P}^n$. 

\begin{lemma}\label{lacosa}
For $n\geq 4$ we have $\rho_{n,k}^*(C_n)=kR_n$.
\end{lemma}

\begin{proof}
Recall that $$C_n=\mathrm{div}_{X_{3,k}}\left(\frac{(n-3)(n-2)}{2}x_1^k-((n-2)^2-1)x_2^k+\frac{(n-2)(n-1)}{2}x_3^k\right).$$ Then by Lemma \ref{tusfi} we have \begin{eqnarray*}\rho_{n,k}^*(C_n)&=&\mathrm{div}_{X_{n,k}}\left(\frac{(n-3)(n-2)}{2}x_1^k-((n-2)^2-1)x_2^k+\frac{(n-2)(n-1)}{2}x_3^k\right)\cr &=&\mathrm{div}_{X_{n,k}}(x_n^k)=kR_n.\end{eqnarray*}
\end{proof}

\begin{lemma}\label{surfa}
Each irreducible component of $X_{n,k}$ has dimension two.
\end{lemma}

\begin{proof}
By Theorem I.7.2 in \cite{H}, since $X_{n,k}$ is the intersection of $n-3$ hypersurfaces in $\mathbb{P}^{n-1}$, we obtain that each irreducible component of $X_{n,k}$ has dimension greater than or equal to two. By Lemma \ref{lacosa}, the morphism $\rho_{n,k}$ is finite and surjective, so each irreducible component of $X_{n,k}$ has dimension at most two.
\end{proof}

%The following observation will allow us to prove that the surfaces $X_{n,k}$ are smooth.
%
%\begin{observation}\label{sdmatrices} Let $\alpha\neq \beta$ be different from $1,2,3$. Then the matrix
%$$\left(\begin{array}{cc}
%{(\alpha-3)(\alpha-2)}x_1^{k-1} & 2((\alpha-2)^2-1)x_2^{k-1}\cr
%{(\beta-3)(\beta-2)}x_1^{k-1} & 2((\beta-2)^2-1)x_2^{k-1}
%\end{array}\right)$$
%has determinant $2{x_1^{k-1}x_2^{k-1}}(\alpha-3)(\beta-3)(\alpha-\beta)\neq 0$, if $x_1x_2\neq 0$. The matrix 
%$$\left(\begin{array}{cc}
%{(\alpha-3)(\alpha-2)}x_1^{k-1} & {(\alpha-2)(\alpha-1)}x_3^{k-1}\cr
%{(\beta-3)(\beta-2)}x_1^{k-1} & {(\beta-2)(\beta-1)}x_3^{k-1}
%\end{array}\right)$$
%has determinant $2{x_1^{k-1}x_3^{k-1}}(\alpha-2)(\beta-2)(\alpha-\beta)\neq 0$, if $x_1x_3\neq 0$; and the matrix
%$$\left(\begin{array}{cc}
%2((\alpha-2)^2-1)x_2^{k-1} & {(\alpha-2)(\alpha-1)}x_3^{k-1}\cr
%2((\beta-2)^2-1)x_2 ^{k-1} & {(\beta-2)(\beta-1)}x_3^{k-1}
%\end{array}\right)$$
%has determinant $2{x_2^{k-1}x_3^{k-1}}(\alpha-1)(\beta-2)(\alpha-\beta)\neq 0$, if $x_2x_3\neq 0$.
%\end{observation}

\begin{lemma}\label{sddsmooth}
For each $n\geq 3$, the variety $X_{n,k}$ is smooth.
\end{lemma}

\begin{proof}
Since $X_{3,k}\cong\mathbb{P}^2$ we know that $X_{3,k}$ is smooth. Now let $n\geq 4$. %By Lemma \ref{tusfi}, we have that $X_{n,k}$ is defined by the equations $g_1,\ldots,g_{n-3}$. %From Proposition \ref{sdjacob},
By Ex.4.2.10 in \cite{Liu}, we only have to show that the Jacobian matrix of the homogeneous equations $g_1,\ldots,g_{n-3}$ of $X_{n,k}$ has maximal rank when evaluated at the point $[x_1:\ldots:x_n]\in X_{n,k}$.  This is the following $(n-3)\times n$ matrix
$$\frac{k}{2}\left(\begin{smallmatrix}%\left(\begin{array}{ccccccc}
2x_1^{k-1} & -6x_2^{k-1} & 6x_3^{k-1} & -2x_4^{k-1} & 0 & \cdots & 0 \cr
6x_1^{k-1} & -16x_2^{k-1} & 12x_3^{k-1} & 0 & -2x_5^{k-1} & \ddots & \vdots \cr
\vdots & \vdots & \vdots & \vdots & \ddots & \ddots & 0\cr
{(n-3)(n-2)}x_1^{k-1} & 2((n-2)^2-1)x_2^{k-1} & {(n-2)(n-1)}x_3^{k-1} & 0 & \cdots & 0 & -2x_n^{k-1}
\end{smallmatrix}\right)$$ %\end{array}$$
%This is proved by induction on $n$ by using Lemma \ref{sdnozero}.
We will prove by induction that this matrix has maximal rank equal to $n-3$.  Suppose that the following $(i-3)\times i$ matrix with $4\leq i\leq n-1$ has maximal rank.
$$M_i=\frac{k}{2}\left(\begin{smallmatrix} %\begin{array}{ccccccc}2x_1^{k-1} & -6x_2^{k-1} & 6x_3^{k-1} & -2x_4^{k-1} & 0 & \cdots & 0 \cr
6x_1^{k-1} & -16x_2^{k-1} & 12x_3^{k-1} & 0 & -2x_5^{k-1} & \ddots & \vdots \cr
\vdots & \vdots & \vdots & \vdots & \ddots & \ddots & 0\cr
{(i-3)(i-2)}x_1^{k-1} & 2((i-2)^2-1)x_2^{k-1} & {(i-2)(i-1)}x_3^{k-1} & 0 & \cdots & 0 & -2x_i^{k-1}
\end{smallmatrix}\right)$$%\end{array$
and consider the $(i-2)\times (i+1)$ matrix
$$M_{i+1}=\frac{k}{2}\left(\begin{smallmatrix} %\begin{array}{ccccccc}
2x_1^{k-1} & -6x_2^{k-1} & 6x_3^{k-1} & -2x_4^{k-1} & 0 & \cdots & 0 \cr
6x_1^{k-1} & -16x_2^{k-1} & 12x_3^{k-1} & 0 & -2x_5^{k-1} & \ddots & \vdots \cr
\vdots & \vdots & \vdots & \vdots & \ddots & \ddots & 0\cr
{(i-2)(i-1)}x_1^{k-1} & 2((i-1)^2-1)x_2^{k-1} & {(i-1)(i)}x_3^{k-1} & 0 & \cdots & 0 & -2x_{i+1}^{k-1}
\end{smallmatrix}\right).$$ %\end{array}\right)$$
If $x_{i+1}\neq 0$ then the matrix $M_{i+1}$ has rank $M_i +1=i-2$. 

Now suppose that $x_{i+1}=0$. By Lemma \ref{sdnozero}, at most one among $x_1,\ldots, x_i$ can be zero. If all $x_j$ with $4\leq j\leq i$ are non-zero, then we are done. If $x_j=0$ for some $4\leq j\leq i$, then we only have to prove that the $(j-3)$-th row is not a multiple of the $(i-2)$-nd row. By Lemma \ref{sdnozero} we have that none of $x_1,x_2,x_3$ is zero. Then the $j$-th row is not a multiple of the first row by analyzing their entries in the first three columns, thus the matrix $M_{i+1}$ has maximal rank $i-2$. 

We know from Lemma \ref{sdnozero} that for $i=4$ the matrix $M_i$ has maximal rank (equal to $1$), therefore the Jacobian matrix of $X_{n,k}$ has maximal rank and hence $X_{n,k}$ is smooth.
\end{proof}

\begin{proposition}\label{sdgentype} For $n\geq 3$, the variety $X_{n,k}$ is a smooth and irreducible surface. It is a complete intersection. The canonical sheaf of $X_{n,k}$ is $\mathcal{O}(k(n-3)-n)$, so $X_{n,k}$ is of general type for $n> \frac{3k}{k-1}$.
\end{proposition}

\begin{proof} Since $X_{n,k}$ is a smooth complete intersection by Lemma \ref{sddsmooth}, we get from \cite{H}, Ex.II.8.4(c), that $X_{n,k}$ is connected, and since it is smooth we obtain that $X_{n,k}$ is irreducible. It is a surface by Lemma \ref{surfa}. From \cite{H}, Ex.II.8.4(e) we have that the canonical sheaf of $X_{n,k}$ is $\mathcal{O}(k(n-3)-n)$. %By Example II.7.6.1 in \cite{H}, the sheaf $\mathcal{O}(k(n-3)-n)$ is ample for $n> \frac{3k}{k-1}$, thus a multiple of $\mathcal{O}(k(n-3)-n)$ determines an embedding $X_{n,k}\to\mathbb{P}^N$, and hence $X_{n,k}$ is a surface of general type, by Theorem V.6.5 in \cite{H}.
\end{proof}

\section{A suitable symmetric differential on $X_{3,k}$}

Let $\left\{U_i\right\}\subseteq \mathbb{P}^2$ be the usual affine open cover of $X_{3,k}=\mathbb{P}^2$ with $U_i=D_+(X_i)$.
Working on the open set $U_3$ with affine coordinates $x_1=\frac{X_1}{X_3}$, $x_2=\frac{X_2}{X_3}$, let us introduce the following symmetric differential 
$$x_1^{2k-2}x_2^2dx_1dx_1+(x_1^{k-1}x_2-4x_1^{k-1}x_2^{k+1}-x_1^{2k-1}x_2)dx_1dx_2+4x_1^kx_2^kdx_2dx_2\in S^2\Omega^1_{U_3/\mathbb{C}}.$$

\begin{proposition}
The previous symmetric differential form in $U_3$ can be extended to a form $$\omega\in H^0(\mathbb{P}^2,\mathcal{O}(2k+3)\otimes S^2\Omega^1_{\mathbb{P}^2/\mathbb{C}}).$$
\end{proposition}

\begin{proof}
In the open set $U_1$ with affine coordinates $x_2:=\frac{X_2}{X_1}$, $x_3:=\frac{X_3}{X_1}$ the section $\omega$ becomes 
\begin{eqnarray*}
&&\frac{x_2^2}{x_3^{2k}}d\frac{1}{x_3}d\frac{1}{x_3}+\left(\frac{x_2}{x_3^k}-4\frac{x_2^{k+1}}{x_3^{2k}}-\frac{x_2}{x_3^{2k}}\right)d\frac{1}{x_3}d\frac{x_2}{x_3}+4\frac{x_2^k}{x_3^{2k}}d\frac{x_2}{x_3}d\frac{x_2}{x_3}\\
&=&\frac{1}{x_3^{2k+3}}(4x_2^kx_3dx_2dx_2+(-x_2x_3^k+x_2-4x_2^{k+1})dx_2dx_3+x_2^2x_3^{k-1}dx_3dx_3).\end{eqnarray*}

Similarly, in the open set $U_2$ with affine coordinates $x_1:=\frac{X_1}{X_2}$, $x_3:=\frac{X_3}{X_2}$, the section $\omega$ becomes
$$
\frac{1}{x_3^{2k+3}}(x_1^{2k-2}x_3dx_1dx_1+(-x_1^{2k-1}-x_1^{k-1}x_3^k+4x_1^{k-1})dx_1dx_3+x_1^kx_3^kdx_3dx_3).
$$
%\begin{proof}
%Write
%\begin{eqnarray*}
%A&=&x_1^2x_2^2dx_1dx_1,\cr
%B&=&(x_1x_2-4x_1x_2^3-x_1^3x_2)dx_1dx_2,\cr
%C&=&4x_1^2x_2^2dx_2dx_2.
%\end{eqnarray*}
%In the open set $U_1$ with affine coordinates $x_2:=\frac{X_2}{X_1}$, $x_3:=\frac{X_3}{X_1}$, these forms become
%\begin{eqnarray*}
%A&=&x_2^2x_3^{-4}d(x_3^{-1})d(x_3^{-1}),\cr
%B&=&(x_2x_3^{-2}-4x_2^3x_3^{-4}-x_2x_3^{-4})d(x_3^{-1})d(x_2x_3^{-1}),\cr
%C&=&4x_2^2x_3^{-4}d(x_2x_3^{-1})d(x_2x_3^{-1}),
%\end{eqnarray*}
%hence in $U_1$
%$$A+B+C=\frac{1}{x_3^7}(4x_2^2x_3dx_2dx_2+(x_2-4x_2^3-x_2x_3^2)dx_2dx_3+x_2^2x_3dx_3dx_3).$$
%Similarly, in the open set $U_2$ with affine coordinates $x_1:=\frac{X_1}{X_2}$, $x_3:=\frac{X_3}{X_2}$, we have
%%\begin{eqnarray*}
%%A&=&x_1^2x_3^{-4}d(x_1x_3^{-1})d(x_1x_3^{-1}),\cr
%%B&=&(x_1x_3^{-2}-4x_1x_3^{-4}-x_1^3x_3^{-4})d(x_1x_3^{-1})d(x_3^{-1}),\cr
%%C&=&4x_1^2x_3^{-4}d(x_3^{-1})d(x_3^{-1}),
%%\end{eqnarray*}
%%hence in $U_2$,
%$$A+B+C=\frac{1}{x_3^7}(x_1^2x_3dx_1dx_1+(-x_1x_3^2-x_1^3+4x_1)dx_1dx_3+x_1^2x_3dx_3dx_3).$$
Therefore our symmetric differential extends to a global section $$\omega\in H^0(\mathbb{P}^2,\mathcal{O}(2k+3)\otimes S^2\Omega^1_{\mathbb{P}^2}).$$
\end{proof}

\begin{remark}
Note that this symmetric differential $\omega$ is substantially different from the ones constructed in \cite{V} and \cite{subm}, because it is constrained by the geometry of the problem under consideration, and it is not intrinsic to the method.
\end{remark}

\begin{remark}\label{elpoli}
Now we will find the complete set of $\omega$-integral curves in $X_{3,k}=\mathbb{P}^2$. For this it will be useful to note that
if $k$ is an even integer, then the polynomial $$Q(x_1,x_2):=x_1^{2k}-8x_1^kx_2^k-2x_1^k+16x_2^{2k}-8x_2^k+1$$ factors in irreducible factors as follows
$$Q=(x_1^{k/2}-2x_2^{k/2}-1)(x_1^{k/2}-2x_2^{k/2}+1)(x_1^{k/2}+2x_2^{k/2}-1)(x_1^{k/2}+2x_2^{k/2}+1),$$
however, when $k$ is odd then $Q(x_1,x_2)$ is irreducible.
\end{remark}

\begin{proposition}\label{sdomega}
The following irreducible curves in $X_{3,k}$ are $\omega$-integral:
\begin{itemize}
\item[(i)] $C_{\alpha}:\frac{(\alpha-3)(\alpha-2)}{2}x_1^k-((\alpha-2)^2-1)x_2^k+\frac{(\alpha-2)(\alpha-1)}{2}x_3^k=0$, $\alpha\in \mathbb{C}\backslash\left\{1,2,3\right\}$;
\item[(ii)] $C_\infty: x_1^k-2x_2^k+x_3^k=0$;
\item[(iii)] The coordinate axes $x_1=0$, $x_2=0$, $x_3=0$;
\item[(iv)] If $k$ is even, the four curves $x_1^{k/2}\pm 2x_2^{k/2}\pm x_3^{k/2}=0$.
\item[(v)] If $k$ is odd, the curve $x_1^{2k}-8x_1^kx_2^k-2x_1^kx_3^k+16x_2^{2k}-8x_2^kx_3^k+x_3^{2k}=0$.
\end{itemize}
\end{proposition}

\begin{proof} 
%For the coordinate axes (curves of type (iii)), it is easy to prove that $\omega$ vanishes along them. (Write $\omega$ in appropriate coordinates for each case.) %For example, on $U_3$ we have
%$$\omega_{|x_1=0}=x_1^2x_2^2dx_1dx_1+(x_1x_2-4x_1x_2^2-x_1^3x_2)dx_1dx_2+4x_1^2x_2^2dx_2dx_2=0$$
%because the curve $x_1=0$ satisfies $dx_1=0$. Hence by Theorem \ref{TYU} they are $\omega$-integral.
The curves of type (i), (ii), (iii) and (iv) are clearly irreducible. The curve of type (v) is irreducible by the previous discussion on the polynomial $Q(x_1,x_2)$.

Let $C_\alpha$ be a curve of type (i). The curve $C_\alpha$ restricted to $U_3$ has equation $x_2^k=\frac{\alpha-1}{2(\alpha-1)}x_1^k+\frac{\alpha-2}{2(\alpha-3)}$. Taking differentials we obtain $dx_2=\frac{\alpha-2}{2(\alpha-1)}\frac{x_1^{k-1}}{x_2^{k-1}}dx_1$.
Thus the restriction of $\omega$ to $C_\alpha$ has equation
$$\begin{aligned}
&x_1^{2k-2}x_2^2dx_1dx_1+(x_1^{k-1}x_2-4x_1^{k-1}x_2^{k+1}-x_1^{2k-1}x_2)dx_1dx_2+4x_1^kx_2^kdx_2dx_2\\
&=x_1^{2k-2}x_2^{2-k}\left(x_2^k+(1-4x_2^k-x_1^k)\frac{\alpha-2}{2(\alpha-1)}+4\left(\frac{\alpha-2}{2(\alpha-1)}\right)^2x_1^k\right)dx_1dx_1\\
&=x_1^{2k-2}x_2^{2-k}\left(\left(1-4\frac{\alpha-2}{2(\alpha-1)}\right)\left(\frac{\alpha-2}{2(\alpha-1)}x_1^k+\frac{\alpha-2}{2(\alpha-3)}\right)\right.\\
&\quad \left.+(1-x_1^k)\frac{\alpha-2}{2(\alpha-1)}+4\left(\frac{\alpha-2}{2(\alpha-1)}\right)x_1^k\right)dx_1dx_1\\
&=0
\end{aligned}$$

Hence by Theorem \ref{TYU}, the curves of type (i) are $\omega$-integral.

It can be similarly checked that curves of type (ii), (iii) and (iv) are $\omega$-integral in the corresponding cases.

Now let $C$ be a curve of type (v). Restricted to $U_3$, this curve has equation $x_1^{2k}-8x_1^kx_2^k-2x_1^k+16x_2^{2k}-8x_2^k+1=0$. Taking differentials, we obtain $$x_1^{k-1}dx_1=\frac{4x_2^{k-1}(x_1^k-4x_2^k+1)}{x_1^k-4x_2^k-1}dx_2.$$
Hence the restriction of $\omega$ to $C$ is
$$\begin{aligned}
&\left(\frac{16x_2^{2k}(x_1^k-4x_2^k+1)^2}{(x_1^k-4x_2^k-1)^2}+\frac{4x_2^{k-1}(x_1^k-4x_2^k+1)}{x_1^k-4x_2^k-1}x_2(1-4x_2^k-x_1^k)+4x_1^kx_2^k\right)dx_2dx_2\\
&=\frac{4x_2^k}{(x_1^k-4x_2^k-1)^2}\left(4x_2^k(x_1^k-4x_2^k+1)^2-(x_1^k-4x_2^k+1)(x_1^k+4x_2^k-1)(x_1^k-4x_2^k-1)\right.\\ &\quad\left.+x_1^k(x_1^k-4x_2^k-1)^2\right)dx_2dx_2\\
&=\frac{-4x_2^k}{(x_1^k-4x_2^k-1)^2}\left(x_1^{2k}-8x_1^kx_2^k-2x_1^k+16x_2^{2k}-8x_2^k+1\right)dx_2dx_2\\ 
&=0\\
\end{aligned}$$
Thus $C$ is $\omega$-integral.
\end{proof}

%The curve of type (ii) restricted to $U_3$ has equation $x_1^2-2x_2^2+1=0$. Taking differentials we get $dx_1=2\frac{x_2}{x_1}dx_2$. Therefore
%\begin{eqnarray*}
%A_{|C_\infty}&=&4x_2^4dx_2dx_2,\cr
%B_{|C_\infty}&=&2(2x_2^2-6x_2^4)dx_2dx_2,\cr
%C_{|C_\infty}&=&4(2x_2^4-x_2^2)dx_2dx_2,
%\end{eqnarray*}
%hence the restriction of $\omega$ to the curve $C_\infty$ has equation
%\begin{eqnarray*}
%&&(4x_2^4+2(2x_2^2-6x_2^4)+4(2x_2^4-x_2^2))dx_2dx_2\cr
%&=& (2x_2^2-4x_2^4+2x_1^2x_2^2)dx_2dx_2\cr
%&=& (x_1^2+1-(x_1^2+1)^2+x_1^2(x_1^2+1))dx_2dx_2\cr
%&=& 0.
%\end{eqnarray*}
%Thus by Theorem \ref{TYU}, the curve of type (ii) is $\omega$-integral.

%
%Let $C_{\epsilon_2,\epsilon_2}:=x_1+2\epsilon_2x_2+\epsilon_3x_3=0$, with $\epsilon_2,\epsilon_3\in\left\{\pm 1\right\}$. In $U_3$ it has equation $x_1=-2\epsilon_2x_2-\epsilon_3$. Taking differentials we get $dx_1=-2\epsilon_2dx_2$. Hence
%\begin{eqnarray*}
%A_{|C_{\epsilon_2,\epsilon_3}}&=&x_1^2x_2^2(-2\epsilon_2)^2dx_2dx_2= (16x_2^4+4x_2^2+16\epsilon_2\epsilon_3x_2^3)dx_2dx_2,\cr
%B_{|C_{\epsilon_2,\epsilon_3}}&=&(x_1x_2-4x_1x_2^3-x_1^3x_2)(-2\epsilon_2)dx_2dx_2=(-32x_2^4-32\epsilon_2\epsilon_3x_2^3-8x_2^2)dx_2dx_2,\cr
%C_{|C_{\epsilon_2,\epsilon_3}}&=&4x_1^2x_2^2dx_2dx_2=(16x_2^4+4x_2^2+16\epsilon_2\epsilon_3x_2^3)dx_2dx_2.
%\end{eqnarray*}
%Thus $\omega_{|C_{\epsilon_2,\epsilon_3}}$ has equation
%\begin{eqnarray*}
%A_{|C_{\epsilon_2,\epsilon_3}}+B_{|C_{\epsilon_2,\epsilon_3}}+C_{|C_{\epsilon_2,\epsilon_3}}=0.
%\end{eqnarray*}
%Therefore from Theorem \ref{TYU}, the curves of type (iv) are $\omega$-integral.

\begin{lemma}
The curves from Proposition \ref{sdomega} are the only $\omega$-integral curves on the surface $X_{3,k}=\mathbb{P}^2$.
\end{lemma}

\begin{proof}
The restriction of $\omega$ to $U_3$ is 
$$x_1^{2k-2}x_2^2dx_1dx_1+(x_1^{k-1}x_2-4x_1^{k-1}x_2^{k+1}-x_1^{2k-1}x_2)dx_1dx_2+4x_1^kx_2^kdx_2dx_2.$$
We have from Definition \ref{delta}
$$\Delta_{U_3}=(\mathbb{P}^2\backslash U_3)\cup\left\{P\in U_3:A_0(P)=0\mbox{ or }A_1^2(P)-4A_0(P)A_2(P)=0\right\}\subseteq\mathbb{P}^2.$$
with
\begin{eqnarray}
A_1^2(P)-4A_0(P)A_2(P)&=&(x_1^{k-1}x_2-4x_1^{k-1}x_2^{k+1}-x_1^{2k-1}x_2)^2-16x_1^{3k-2}x_2^{k+2}\cr
&=&x_1^{2k-2}x_2^2(x_1^{2k}-8x_1^kx_2^k-2x_1^k+16x_2^{2k}-8x_2^k+1).\cr
\end{eqnarray}
This factors as follows when $k$ is even:
\begin{eqnarray}\label{di}
A_1^2(P)-4A_0(P)A_2(P)&=&x_1^{2k-2}x_2^2\prod_{\epsilon_2,\epsilon_3\in\{-1,1\}} (x_1^{k/2}+\epsilon_2x_2^{k/2}+\epsilon_3).
\end{eqnarray}

We also have $A_0=x_1^{2k-2}x_2^2$. Therefore $\Delta_{U_3}$ is an union of $\omega$-integral curves.
From Theorem \ref{eee} we only need to prove that for any $P\notin\Delta_{U_3}$ there are at least two $\omega$-integral curves passing through $P$. Let $P=[x_1:x_2:x_3]$ be a point outside $\Delta_{U_3}$. The point $P$ lies on a curve $C_\alpha$ if and only if $$\frac{(\alpha-3)(\alpha-2)}{2}x_1^k-((\alpha-2)^2-1)x_2^k+\frac{(\alpha-2)(\alpha-1)}{2}x_3^k=0,$$ and it is in $C_\infty$ if and only if $x_1^k-2x_2^k+x_3^k=0$. Since $P\notin \Delta_{U_3}$, we have that $x_i\neq 0$ for $i=1,2,3$ (because $\Delta_{U_3}$ contains the coordinate axes). 
The discriminant of the quadratic equation (with $\alpha$ the unknown)
\begin{eqnarray}\label{alfa}
\frac{(\alpha-3)(\alpha-2)}{2}x_1^k-((\alpha-2)^2-1)x_2^k+\frac{(\alpha-2)(\alpha-1)}{2}x_3^k=0
\end{eqnarray}
is
\begin{eqnarray*}
&&\left(-\frac{3}{2}x_3^k+4x_2^k-\frac{5}{2}x_1^k\right)^2-4\left(x_3^k-3x_2^k+3x_1^k\right)\left(\frac{x_3^k}{2}-x_2^k+\frac{x_1^k}{2}\right)\cr
&=& \frac{1}{4}(x_1^{2k}-8x_1^kx_2^k-2x_1^kx_3^k+16x_2^{2k}-8x_2^kx_3^k+x_3^{2k}),
\end{eqnarray*}
which is different from zero at $P$ because it is outside $\Delta_{U_3}$ (by Equation \eqref{di}). The leading coefficient with respect to $\alpha$ from the quadratic Equation \eqref{alfa} is $\frac{1}{2}(x_1^k-2x_2^k+x_3^k)$. Suppose that $P$ is not in $C_\infty$, then Equation \eqref{alfa} has degree $2$ in $\alpha$. Noting that Equation \eqref{alfa} is the equation of $C_\alpha$ (see Proposition \ref{sdomega}), we get that there are precisely two values of $\alpha$ such that the curve $C_\alpha$ passes through $P$. On the other hand, if $P\in C_\infty$, then Equation \eqref{alfa} is linear in $\alpha$ and it has exactly one solution $\alpha_0$ and we get that $C_{\infty}$ and $C_{\alpha_0}$ pass through $P$. Since this holds for every $P$ outside $\Delta_{U_3}$ we know from Theorem \ref{eee} that there are no more $\omega$-integral curves in $X_{3,k}$.
\end{proof}

\section{Pullbacks of curves of type $(i)$, $(ii)$ and $(iii)$}

%Recall that $\rho_{n,k}=\pi_4\circ\ldots\circ\pi_n$. 
From Proposition \ref{sdomega}, we have the complete list of $\omega$-integral curves of $X_{3,k}\cong\mathbb{P}^2$. Now we will compute their pullbacks to the surfaces $X_{n,k}$ for greater values of $n$, so we can later know all $\rho_{n,k}^\bullet\omega$-integral curves on the surfaces $X_{n,k}$. In this section, we will study the pullbacks of curves of type (i), (ii) and (iii).  %We will use this list to find all $\rho_{n,k}^{\bullet}\omega$-integral curves on the other surfaces $X_{n,k}$. 

\begin{lemma}\label{sdrediii}
The pull-back under $\rho_{n,k}$ of a curve of type (iii) of Proposition \ref{sdomega} is a smooth complete intersection curve.\end{lemma}

\begin{proof}
Let $C$ be the pullback of a curve of type (iii). It is given by the equations
\begin{eqnarray*}
x_i&=&0\cr
x_1^k-3x_2^k+3x_3^k&=&x_4^k\cr
&\vdots&\cr
\frac{(n-3)(n-2)}{2}x_1^k-((n-2)^2-1)x_2^k+\frac{(n-2)(n-1)}{2}x_3^k&=&x_n^k,
\end{eqnarray*}
with $i\in\left\{1,2,3\right\}$.
Since $C$ is a curve defined by $n-2$ equations we get that it is a complete intersection.
%The equations of $C$ are 
%\begin{eqnarray*}
%x_i&=&0\cr
%x_1^2-3x_2^2+3x_3^2&=&x_4^2\cr
%&\vdots&\cr
%\frac{(n-3)(n-2)}{2}x_1^2-((n-2)^2-1)x_2^2+\frac{(n-2)(n-1)}{2}x_3^2&=&x_n^2,
%\end{eqnarray*}
%with $i$ one of $0,1,2$, hence it is a complete intersection on $\mathbb{P}^{n-1}$. 
Its Jacobian matrix evaluated at the point $[x_1:\ldots :x_n]\in X_{n,k}$ is a $(n-2)\times n$ matrix of the form
$$\left(\begin{smallmatrix}%{ccccccc}
a & b & c & 0 & 0 & 0 & 0\cr
kx_1^{k-1} & -3kx_2^{k-1} & 3kx_3^{k-1} & -kx_4^{k-1} & 0 & \cdots & 0 \cr
3kx_1^{k-1} & -8kx_2^{k-1} & 6kx_3^{k-1} & 0 & -kx_5^{k-1} & \ddots & \vdots \cr
\vdots & \vdots & \vdots & \vdots & \ddots & \ddots & 0\cr
\frac{(n-3)(n-2)}{2}kx_1^{k-1} & ((n-2)^2-1)kx_2^{k-1} & \frac{(n-2)(n-1)}{2}kx_3^{k-1} & 0 & \cdots & 0 & -kx_n^{k-1}
\end{smallmatrix}\right)$$
with exactly one of $a,b,c$ different from zero (and equal to $1$). 
By a reasoning similar to the proof of Lemma \ref{sddsmooth}, this $(n-2)\times n$ matrix will have maximal rank. %The cases $x_2=0$ and $x_3=0$ are proved similarly.
\end{proof}

\begin{lemma}\label{sdredi}
The pull-back under $\rho_{n,k}$ of the curve $C_\alpha$ of type (i) of Proposition \ref{sdomega} with $\alpha\neq 4,\ldots,n$ is a smooth complete intersection. The pull-back under $\rho_{n,k}$ of the curve $C_{\infty}$ of type (ii) is a smooth complete intersection.
\end{lemma}

\begin{proof}
Let $C=C_\alpha$ be the pullback of a curve of type (i) with $\alpha\neq 4,\ldots,n$. It is given by the system of equations
\begin{eqnarray*}
\frac{(\alpha-3)(\alpha-2)}{2}x_1^k-((\alpha-2)^2-1)x_2^k+\frac{(\alpha-2)(\alpha-1)}{2}x_3^k&=&0\cr
x_1^k-3x_2^k+3x_3^k&=&x_4^k\cr
&\vdots&\cr
\frac{(n-3)(n-2)}{2}x_1^k-((n-2)^2-1)x_2^k+\frac{(n-2)(n-1)}{2}x_3^k&=&x_n^k.
\end{eqnarray*}
Therefore it is a complete intersection. %By Proposition \ref{pullback} the equations of $C$ are
%\begin{eqnarray*}
%\frac{(\alpha-3)(\alpha-2)}{2}x_1^2-((\alpha-2)^2-1)x_2^2+\frac{(\alpha-2)(\alpha-1)}{2}x_3^2&=&0\cr
%x_1^2-3x_2^2+3x_3^2&=&x_4^2\cr
%&\vdots&\cr
%\frac{(n-3)(n-2)}{2}x_1^2-((n-2)^2-1)x_2^2+\frac{(n-2)(n-1)}{2}x_3^2&=&x_n^2
%\end{eqnarray*}
%
%Therefore, this curve is a complete intersection in $\mathbb{P}^{n-1}$. 
The Jacobian matrix of $C$ evaluated at $[x_1:\ldots :x_n]\in X_{n,k}$ is the $(n-2)\times n$ matrix
$$\frac{k}{2}\left(\begin{smallmatrix}%{ccccccc}
(\alpha-3)(\alpha-2)x_1^{k-1} & 2((\alpha-2)^2-1)x_2^{k-1} & (\alpha-2)(\alpha-1)x_3^{k-1} & 0 & 0 & 0 & 0\cr
2x_1^{k-1} & -6x_2^{k-1} & 6x_3^{k-1} & -2x_4^{k-1} & 0 & \cdots & 0 \cr
6x_1^{k-1} & -16x_2^{k-1} & 12x_3^{k-1} & 0 & -2x_5^{k-1} & \ddots & \vdots \cr
\vdots & \vdots & \vdots & \vdots & \ddots & \ddots & 0\cr
{(n-3)(n-2)}x_1^{k-1} & 2((n-2)^2-1)x_2^{k-1} & {(n-2)(n-1)}x_3^{k-1} & 0 & \cdots & 0 & -2x_n^{k-1}
\end{smallmatrix}\right).$$

A computation similar to the proof of Lemma \ref{sddsmooth} gives us that the curve $C_\alpha$ is smooth.
%
%The equations defining the curve of type (ii) are
%\begin{eqnarray*}
%x_1^2-2x_2^2+x_3^2&=&0\cr
%x_1^2-3x_2^2+3x_3^2&=&x_4^2\cr
%&\vdots&\cr
%\frac{(n-3)(n-2)}{2}x_1^2-((n-2)^2-1)x_2^2+\frac{(n-2)(n-1)}{2}x_3^2&=&x_n^2.
%\end{eqnarray*}
%Hence this curve is a complete intersection in $\mathbb{P}^{n-1}$. 

The curve $C_\infty$ is given by the equations
\begin{eqnarray*}
x_1^k-2x_2^k+x_3^k&=&0\cr
x_1^k-3x_2^k+3x_3^k&=&x_4^k\cr
&\vdots&\cr
\frac{(n-3)(n-2)}{2}x_1^k-((n-2)^2-1)x_2^k+\frac{(n-2)(n-1)}{2}x_3^k&=&x_n^k.
\end{eqnarray*}
The analysis for this curve is similar.
%Therefore it is a complete intersection.
%The Jacobian matrix of $C_{\infty}$ evaluated at $[x_1:\ldots :x_n]\in X_{n,k}$ is the $(n-2)\times n$ matrix
%$$\frac{k}{2}\left(\begin{smallmatrix} %{ccccccc}
%2x_1^{k-1} & -4x_2^{k-1} & 2x_3^{k-1} & 0 & 0 & 0 & 0\cr
%2x_1^{k-1} & -6x_2^{k-1} & 6x_3^{k-1} & -2x_4^{k-1} & 0 & \cdots & 0 \cr
%6x_1^{k-1} & -16x_2^{k-1} & 12x_3^{k-1} & 0 & -2x_5^{k-1} & \ddots & \vdots \cr
%\vdots & \vdots & \vdots & \vdots & \ddots & \ddots & 0\cr
%{(n-3)(n-2)}x_1^{k-1} & 2((n-2)^2-1)x_2^{k-1} & {(n-2)(n-1)}x_3^{k-1} & 0 & \cdots & 0 & -2x_n^{k-1}
%\end{smallmatrix}\right).$$
%The submatrix 
%$$\left(\begin{array}{cc}
%2x_1 & -4x_2\cr
%{(i-3)(i-2)}x_1 & 2((i-2)^2-1)x_2
%\end{array}\right)$$
%has determinant $4(i-3)(2i-3)x_1x_2\neq 0$ when $i\neq 3$ and $x_1x_2\neq 0$.
%The submatrix
%$$\left(\begin{array}{cc}
%2x_1 & 2x_3\cr
%{(i-3)(i-2)}x_1 & {(i-2)(i-1)}x_3
%\end{array}\right)$$
%has determinant $4(i-2)x_1x_3\neq 0$ when $i\neq 2$ and $x_1x_3\neq 0$.
%The submatrix 
%$$\left(\begin{array}{cc}
%-4x_2 & 2x_3\cr
%2((i-2)^2-1)x_2 & {(i-2)(i-1)}x_3
%\end{array}\right)$$
%has determinant $-4(2i-5)(i-1)x_2x_3\neq 0$ when $i\neq 1$ and $x_2x_3\neq 0$.
%which has maximal rank. Therefore the curve $C_{\infty}$ of type (ii) is smooth.
\end{proof}

\begin{lemma}\label{sdred2}
For $3\leq i\leq n$, the curves $(\pi_{i+1}\circ\cdots\circ\pi_n)^*R_i$ are smooth complete intersections, with $R_i$ as defined in Proposition \ref{rir}.
\end{lemma}

%\begin{proof} By Proposition \ref{pullback} the equations of $(\pi_{i+1}\circ\cdots\circ\pi_n)^*R_i$ are
%\begin{eqnarray*}
%x_i&=&0\cr 
%x_1^2-3x_2^2+3x_3^2&=&x_4^2\cr
%&\vdots&\cr
%\frac{(n-3)(n-2)}{2}x_1^2-((n-2)^2-1)x_2^2+\frac{(n-2)(n-1)}{2}x_3^2&=&x_n^2.
%\end{eqnarray*}
%Therefore $R_i$ is a complete intersection in $\mathbb{P}^{n-1}$. 
\begin{proof} The curve $(\pi_{i+1}\circ\cdots\circ\pi_n)^*R_i$ is given by the system of equations
\begin{eqnarray*}
x_i&=&0\cr 
x_1^k-3x_2^k+3x_3^k&=&x_4^k\cr
&\vdots&\cr
\frac{(n-3)(n-2)}{2}x_1^k-((n-2)^2-1)x_2^k+\frac{(n-2)(n-1)}{2}x_3^k&=&x_n^k.
\end{eqnarray*}
Therefore it is a complete intersection.

The Jacobian matrix of this curve evaluated at the point $[x_0:\ldots :x_n]\in X_{n,k}$ is the $(n-2)\times n$ matrix
$$\frac{k}{2}\left(\begin{smallmatrix}%{cccccccc}
0 & 0 & 0 & 0 & \cdots & 2/k & 0\cr
2x_1^{k-1} & -6x_2^{k-1} & 6x_3^{k-1} & 2x_4^{k-1} & 0 & \cdots & 0 \cr
6x_1^{k-1} & -16x_2^{k-1} & 12x_3^{k-1} & 0 & 2x_5^{k-1} & \ddots  & \vdots \cr
\vdots & \vdots & \vdots & \vdots & \ddots & \ddots  & 0\cr
{(n-3)(n-2)}x_1^{k-1} & 2((n-2)^2-1)x_2^{k-1} & {(n-2)(n-1)}x_3^{k-1} & 0 & \cdots & 0 & 2x_n^{k-1}
\end{smallmatrix}\right)$$
where the non-zero component on the first row is on the $i$-th column, and $2x_i^{k-1}=0$. 
Interchange the $(i-2)$-nd row and the first row. Then a computation similar to the proof of Lemma \ref{sddsmooth} gives that this matrix has maximal rank, hence the curve $R_i$ is smooth.
\end{proof}
We conclude:
\begin{lemma}\label{sdirred}
The pullbacks under $\rho_{n,k}$ of curves of type (i) with $\alpha\neq 4,\ldots,n$, of type (ii) and of type (iii) are irreducible and reduced. The pullback of curves of type (i) with $\alpha=4,\ldots,n$ is $\rho_{n,k}^*(C_\alpha)=kR_\alpha$, where $R_\alpha$ are irreducible and reduced. %If $k$ is even, then the pullbacks of the curves of type (iv) is the union of the curves $C^k_{\bar{\epsilon}}$ with $\bar{\epsilon}\in G$ which are irreducible and reduced. If $k$ is odd, the pullback of the curve of type (v) is irreducible. 
%5Let $C$ be a curve of type (i), (ii) or (iii) in Lemma \ref{sdomega}. Then $(\rho_{n,k}^*(C))_{\mathrm{red}}$ is smooth and irreducible for every $n\geq 4$.
\end{lemma}

\begin{proof} 
The pullbacks of the curves of type (i) with $\alpha\neq 4,\ldots n$, of type (ii) and of type (iii) are smooth and complete intersections by Lemma \ref{sdrediii} and Lemma \ref{sdredi}. %From the proofs of these Lemmas, we know that the pullbacks $\rho_{n,k}^*(C_\alpha)$ of curves of type (i) (with $\alpha\neq 4,\ldots,n$), (ii) and (iii) are smooth complete intersections. 
From \cite{H}, Ex.II.8.4(c) we have that these curves are connected, and therefore irreducible.
By Lemma \ref{lacosa} we have that $$\rho_{n,k}^*C_i=(\pi_{i+1}\circ \cdots\circ \pi_n)^*\rho_{i,k}^*C_i=k(\pi_{i+1}\circ\cdots\circ \pi_n)^*R_i$$ when $4\leq i\leq n$. Since the pullbacks $(\pi_{i+1}\circ\cdots\circ\pi_n)^*R_i$ are smooth and complete intersections by Lemma \ref{sdred2}, we obtain that they are irreducible by \cite{H}, Ex.II.8.4(c). Thus $(\rho_{n,k}^*(C_i))_{\mathrm{red}}=(\pi_{i+1}\circ\cdots\circ\pi_n)^*R_i$ is smooth and irreducible.
%From Lemma \ref{tipo4}, we know that the curves of type (iv) are irreducible and reduced. By Lemma \ref{jeloucito}, we have then that the curve of type (v) is irreducible and reduced.
\end{proof}

\section{Pullbacks of curves of type $(iv)$ and $(v)$}

We now study the pullbacks of the curves of type (iv), when $k$ is an even number, and of curves of type (v), when $k$ is odd.

For $k$ an even integer, let $C^k_{\epsilon_2,\ldots,\epsilon_n}\subseteq \mathbb{P}^{n-1}$ be defined by 
\begin{eqnarray*}
-x_1^{k/2}+2\epsilon_2x_2^{k/2}&=&\epsilon_3x_3^{k/2}\cr
&\vdots&\cr
-(n-2)x_1^{k/2}+(n-1)\epsilon_2x_2^{k/2}&=&\epsilon_nx_n^{k/2}.
\end{eqnarray*}

\begin{remark}
Note that for $k=2$, the curve $C^2_{\epsilon_2,\ldots,\epsilon_n}$ is the line parametrized by $$[s+t:\epsilon_2(2s+t):\ldots:\epsilon_n(ns+t)]$$ with $[s:t]\in\mathbb{P}^1$.
\end{remark}

It will be convenient to write $G_{n}$ for the multiplicative group $\left\{\pm 1\right\}^{n-1}$, and for $\bar{\epsilon}=(\epsilon_2,\ldots,\epsilon_n)\in G_n$ we write $C^k_{\bar{\epsilon}}=C^k_{\epsilon_2,\ldots,\epsilon_n}$.

\begin{lemma}\label{notiene}
For $k$ even and $\bar{\epsilon}\in G_n$, the scheme $C^k_{\bar{\epsilon}}$ is a smooth irreducible curve contained in $X_{n,k}$. These curves satisfy $\pi_n(C^k_{\epsilon_2,\ldots,\epsilon_n})=C^k_{\epsilon_2,\ldots,\epsilon_{n-1}}$. In addition, if $\bar{\epsilon}\neq \bar{\epsilon}'$, then $C^k_{\bar{\epsilon}}\neq C^k_{\bar{\epsilon}'}$.\end{lemma}

\begin{proof} Let us first check that $C^k_{\epsilon_2,\ldots,\epsilon_n}\subseteq X_{n,k}$. From the first equation defining $C^k_{\epsilon_2,\ldots,\epsilon_n}$ we deduce $$x_1^k-4\epsilon_2x_1^{k/2}x_2^{k/2}+4x_2^k=x_3^k,$$
and more generally, the $(j-2)$-th equation gives 
$$(j-2)^2x_1^k-2(j-2)(j-1)\epsilon_2x_1^{k/2}x_2^{k/2}  +(j-1)^2x_2^k =x_j^k.$$
Hence for every $3\leq j\leq n$, the points of $C^k_{\epsilon_2,\ldots,\epsilon_n}$ satisfy
%%% HASTA ACA
%the first equation of $C^k_{\epsilon_2,\ldots,\epsilon_n}$ we have $x_1^k+4x_2^k+4\epsilon_2x_1^{k/2}x_2^{k/2}=x_3^k$. Squaring $-(n-2)x_1^{k/2}+(n-1)\epsilon_2x_2^{k/2}=\epsilon_nx_n^{k/2}$ and replacing $4\epsilon_2x_1^{k/2}x_2^{k/2}$ by $x_3^k-x_1^k-4x_2^k$ we obtain
$$\frac{(j-3)(j-2)}{2}x_1^k-((j-2)^2-1)x_2^k+\frac{(j-2)(j-1)}{2}x_3^k=x_j^k,$$ and so $C^k_{\epsilon_2,\ldots,\epsilon_n}$ is in $X_{n,k}$.

%We have $\pi_n(C^k_{\epsilon_2,\ldots,\epsilon_n})\subseteq C^k_{\epsilon_2,\ldots,\epsilon_{n-1}}$. Since for any $P=[x_1:\ldots:x_{n-1}]\in C^k_{\epsilon_2,\ldots,\epsilon_{n-1}}$  we have that $Q=[x_0:\ldots:x_{n-1}]\in C^k_{\epsilon_2,\ldots,\epsilon_n}$ with $\epsilon_nx_n^{k/2}=-(n-2)x_1^{k/2}+(n-1)\epsilon_2x_2^{k/2}$ is a preimage of $P$, we obtain that $\pi_n(C^k_{\epsilon_2,\ldots,\epsilon_n})\supseteq C^k_{\epsilon_2,\ldots,\epsilon_{n-1}}$, thus
It is easy to check that $\pi_n(C^k_{\epsilon_2,\ldots,\epsilon_n})=C^k_{\epsilon_2,\ldots,\epsilon_{n-1}}$. Since $C^k_{\epsilon_2,\epsilon_3}\subseteq\mathbb{P}^2$ is a smooth irreducible curve, one sees by induction (using the maps $\pi_n$) that $C^k_{\epsilon_2,\ldots,\epsilon_n}$ has all its irreducible components of dimension one. In particular, it is a complete intersection in $\mathbb{P}^{n-1}$.
%From this, we also get that every component of $C^k_{\epsilon_2,\ldots,\epsilon_n}$ has dimension less than or equal to $1$ since $\pi_n(C^k_{\epsilon_2,\ldots,\epsilon_n})\neq \pi_n(X_{n,k})=X_{n-1,k}$. Every irreducible component of $C^k_{\epsilon_2,\ldots,\epsilon_n}$ has dimension at least one because it is defined by $n-2$ equations on $\mathbb{P}^{n-1}$, and hence $C^k_{\epsilon_2,\ldots,\epsilon_n}$ is a complete intersection. 

The Jacobian matrix of the equations defining $C^k_{\epsilon_2,\ldots,\epsilon_n}$ evaluated at a point $[x_1:\ldots:x_n]$ is the following matrix: 
$$\frac{k}{2}\left(\begin{smallmatrix}
-x_1^{\frac{k}{2}-1} & 2\epsilon_2x_2^{\frac{k}{2}-1} & -\epsilon_3x_3^{\frac{k}{2}-1} & 0 & \cdots & 0\cr
-2x_1^{\frac{k}{2}-1} & 3\epsilon_2x_2^{\frac{k}{2}-1} & 0 & -\epsilon_4x_4^{\frac{k}{2}-1} & \ddots & \vdots\cr
\vdots & \vdots & \vdots & \ddots & \ddots & 0\cr
-(n-2)x_1^{\frac{k}{2}-1} & (n-1)\epsilon_2x_2^{\frac{k}{2}-1} & 0 & \cdots & 0 & -\epsilon_nx_n^{\frac{k}{2}-1}
\end{smallmatrix}\right).$$
As in the proof of Lemma \ref{sddsmooth}, one verifies that it has maximal rank at every point of $C^k_{\epsilon_2,\ldots,\epsilon_n}$. Since $C^k_{\epsilon_2,\ldots,\epsilon_n}$ is a smooth complete intersection, we obtain that it is irreducible.

Finally, if $\bar{\epsilon}\neq\bar{\epsilon}'$, the curves $C^k_{\bar{\epsilon}}$ and $C^k_{\bar{\epsilon}'}$ have different images under the linear projection $[x_1:\ldots:x_n]\mapsto[x_1:x_2:x_j]$ for suitable choice of $j$.
\end{proof}

%\begin{lemma}\label{distinto}
%If $\epsilon_i\neq \epsilon'_i$ for some $2\leq i\leq n$, then the curves $C^k_{\epsilon_2,\ldots,\epsilon_n}$ and $C^k_{\epsilon_2',\ldots,\epsilon_n'}$ are distinct.
%\end{lemma}

%\begin{proof}
%Let $i'=i$. If $i=1$, and $i'=2$ if $i=1$. The linear projection $\mathbb{P}^{n-1}\dashrightarrow \mathbb{P}^2$, $[x_1:\ldots:x_n]\mapsto [x_1:x_2:x_3]$ maps $C^k_{\epsilon_2,\ldots,\epsilon_n}$ and $C^k_{\epsilon_2',\ldots,\epsilon_n'}$ to two different Fermat type curves.
%Suppose that $\epsilon_2=\epsilon_2'$. Let $\zeta$ be a primitive $\frac{k}{2}$-th root of $-1$. For $3\leq j\leq n$ let $$\alpha_j=\left\{\begin{array}{cl}
%0 & \mbox{if }\epsilon_2\epsilon_j=1\cr
%1 & \mbox{if }\epsilon_2\epsilon_j=-1.\end{array}\right.$$
%Then the point $P_{\epsilon_2,\ldots,\epsilon_n}:=[0:1:\sqrt[k/2]{2}\zeta^{\alpha_3}:\cdots:\sqrt[k/2]{n-1}\zeta^{\alpha_n}]$ will satisfy the equations $-(n-2)x_1^{k/2}+(n-1)\epsilon_2x_2^{k/2}=\epsilon_nx_n^{k/2}$ for $3\leq j\leq n$, but it will not satisfy $(i-2)x_1^{k/2}+(i-1)\epsilon_2'x_2^{k/2}=\epsilon_i'x_i^{k/2}$.
%If $\epsilon_j=\epsilon_j'$ for some $3\leq j\leq n$, a similar argument holds.
%Now suppose that $\epsilon_j=-\epsilon_j'$ for all $2\leq j\leq n$. Let $\zeta$ be a primitive $\frac{k}{2}$-th root of $-1$. Then the point $[1:0:\sqrt[k/2]{\epsilon_3}\zeta:\cdots:\sqrt[k/2]{(n-2)\epsilon_n}\zeta]$ is a point in $C^k_{\epsilon_2,\ldots,\epsilon_n}$ which is not in $C^k_{\epsilon_2',\ldots,\epsilon_n'}$.
%\end{proof}

\begin{lemma}\label{tipo4} 
Let $k$ be an even integer, and let $\epsilon_2,\epsilon_3\in\left\{\pm 1\right\}$. Then we have the following equality of divisors in $X_{n,k}$:  $$\rho_{n,k}^*C^k_{\epsilon_2,\epsilon_3}=\sum_{\epsilon_4,\ldots,\epsilon_n\in \left\{\pm 1\right\}}C^k_{\epsilon_2,\ldots,\epsilon_n}.$$
%The pullbacks under $\rho_{n,k}$ of the curves of type (iv), QQQQ  $-x_1^{k/2}+2\epsilon_2 x_2^{k/2}=\epsilon_3x_3^{k/2}$ are the $2^{n-1}$ smooth complete intersection curves $C^k_{\epsilon_2,\ldots,\epsilon_n}$. %$$[s+t:\pm(2s+t):\cdots :\pm(ns+t)]\mbox{ with }[s:t]\in\mathbb{P}^1.$$ These are smooth irreducible curves of genus $0$.
\end{lemma}

\begin{proof} Recall from Lemma \ref{notiene} that $\pi_j(C^k_{\epsilon_2,\ldots,\epsilon_j})=C^k_{\epsilon_2,\ldots,\epsilon_{j-1}}$ for all $4\leq j\leq n$. Moreover, from the equations defining these curves, we see that the map $\pi_{j|C^k_{\epsilon_2,\ldots,\epsilon_j}}$ is finite of degree $k/2$.
Therefore $\rho_{n,k}(C^k_{\epsilon_2,\ldots,\epsilon_n})=C^k_{\epsilon_2,\epsilon_3}$ and $$(\rho_{n,k})_{|C^k_{\epsilon_2,\ldots,\epsilon_n}}:C^k_{\epsilon_2,\ldots,\epsilon_n}\to C^k_{\epsilon_2,\epsilon_3}$$ is finite of degree $(k/2)^{n-2}$. 
%For fixed $\epsilon_2,\epsilon_3$, the image of any $C^k_{\epsilon_2,\ldots,\epsilon_n}$ under $\rho_{n,k}$ is $C^k_{\epsilon_2,\epsilon_3}$. Now we want to prove that these irreducible curves are all the preimages of $C^k_{\epsilon_2,\epsilon_3}$. The restriction of the morphism $\mathbb{P}^n\setminus[0:\ldots:0:1]\to\mathbb{P}^{n-1}$ to $C^k_{\epsilon_2,\ldots,\epsilon_n}$ gives a morphism $C^k_{\epsilon_2,\ldots,\epsilon_n}\to C^k_{\epsilon_2,\ldots,\epsilon_{n-1}}$ which has degree $\frac{k}{2}$. The composition of these restrictions has degree $\left(\frac{k}{2}\right)^{n-2}$, hence $\deg(\rho_{n,k|C^k_{\epsilon_2,\ldots,\epsilon_n}})=\left(\frac{k}{2}\right)^{n-2}$. 

From Lemma \ref{notiene}, all the curves $C^k_{\epsilon_2,\ldots,\epsilon_n}$ for $\epsilon_4,\ldots,\epsilon_n\in\{\pm 1\}$ are distinct. %all the components of the preimage of $C^k_{\epsilon_2,\epsilon_3}$ are distinct.
Thus there are $2^{n-2}$ distinct curves $C^k_{\epsilon_2,\ldots,\epsilon_n}$ in the preimage of $C^k_{\epsilon_2,\epsilon_3}$. 
At this point, we know that
$$
\rho_{n,k}^*C^k_{\epsilon_2,\epsilon_3}\ge \sum_{\epsilon_4,\ldots,\epsilon_n\in \left\{\pm 1\right\}}C^k_{\epsilon_2,\ldots,\epsilon_n}.
$$
The previous inequality of divisors is actually an equality because a general point in $C^k_{\epsilon_2,\epsilon_3}$ has exactly $k^{n-2}$ preimages by $\rho_{n,k}$ on each side of the inequality.% in the scheme associated
\end{proof}

%\begin{proof} First note that $[s+t:\pm(2s+t):\ldots: \pm(ns+t)]$ is in $X_n$ because taking squares we obtain a sequence consisting of squares of elements in arithmetic progression (which have constant second differences). 
%
%The image under $\pi_n$ of the curve $[s+t:\epsilon_2(2s+t):\ldots: \epsilon_n(ns+t)]$ (with $\epsilon_i\in\left\{\pm 1\right\}$) is the curve $[s+t:\epsilon_2(2s+t):\ldots: \epsilon_{n-1}((n-1)s+t)]$. Since $\pi_n$ is of degree $2$ and there are two different curves mapping onto $[s+t:\epsilon_2(2s+t):\ldots: \epsilon_{n-1}((n-1)s+t)]$, namely $[s+t:\epsilon_2(2s+t):\ldots: (ns+t)]$ and $[s+t:\epsilon_2(2s+t):\ldots: -(ns+t)]$, we obtain from Proposition \ref{ordg} that these two curves are all the preimage curves of $[s+t:\epsilon_2(2s+t):\ldots: \epsilon_{n-1}((n-1)s+t)]$.
%
%Write $C_{\epsilon_2,\ldots,\epsilon_n}$ for the curve in $X_n$ given by $[s+t:\epsilon_2(2s+t):\ldots: \epsilon_n(ns+t)]$ and $\epsilon_i\in\left\{\pm 1\right\}$.
%The image under $\rho_{n,k}$ of $C_{\epsilon_2,\ldots,\epsilon_n}$ is $[s+t:\epsilon_2 (2s+t):\epsilon_3 (3s+t)]$. The points on this curve satisfy the equation $x_1-\epsilon_2 2x_2+\epsilon_3 x_3=0$, which is the equation of one of the curves of type (iv).
%The curves $C_{\epsilon_2,\ldots,\epsilon_n}$ are irreducible smooth curves of genus zero because they are isomorphic images of $\mathbb{P}^1$ (the inverse of this map is the morphism $\eta:C_{\epsilon_2,\ldots,\epsilon_n}\to \mathbb{P}^1$, mapping $[x_1:x_2:\cdots:x_n]$ to $[\epsilon_2x_2-x_1:2x_1+\epsilon_2x_2]$).
%\end{proof}

For $k$ odd, let $C^{(v)}_{n,k}$ be the pull-back under $\rho_{n,k}$ to $X_{n,k}$ of the curve of type $(v)$ of $X_{3,k}$. We have the following:

\begin{lemma}\label{jeloucito}
Let $k$ be odd. Then $C^{(v)}_{n,k}\subseteq X_{n,k}$ is a reduced and irreducible curve. %, and is given by the equations 
%\begin{eqnarray*}
%x_1^{2k}-8x_1^kx_2^k-2x_1^kx_3^k+16x_2^{2k}-8x_2^kx_3^k+x_3^{2k}&=&0\cr
%x_1^k-3x_2^k+3x_3^k&=&x_4^k\cr
%&\vdots&\cr
%x_{n-3}^k-3x_{n-2}^k+3x_{n-1}^k&=&x_n^k.
%\end{eqnarray*}
Moreover, the $2^n$ curves $C^{2k}_{\epsilon_2,\ldots,\epsilon_n}\subseteq X_{n,2k}$ are isomorphic to each other, and are birational to $C^{(v)}_{n,k}$.
\end{lemma}

\begin{proof}
Let $F_n:X_{n,2k}\to X_{n,k}$ be the morphism given by $[x_1:\ldots:x_n]\mapsto [x_1^2:\ldots:x_n^2]$. This morphism is finite of degree $2^{n-1}$.

The group $G_n=\{\pm 1\}^{n-1}$ has order $2^{n-1}$ and since $k$ is odd, it acts faithfully on $X_{n,2k}$ by $\tau (X_{n,2k})=X_{n,2k}$. Moreover, this action satisfies $F_n(\tau\cdot [x_1:\ldots:x_n])=F_n([x_1:\ldots:x_n])$. So $F_n:X_{n,2k}\to X_{n,k}$ is Galois with Galois group $G_n$.
We have the commutative diagram
$$\begin{CD}
X_{n,2k} @>{F_n}>> X_{n,k}\\
@V{\rho_{n,2k}}VV @VV{\rho_{n,k}}V\\
X_{3,2k} @>>{F_3}> X_{3,k}
\end{CD}$$
The curve $C_{3,k}^{(v)}\subseteq X_{3,k}$ is defined by the homogenization of the polynomial $Q(x_1,x_2)$ from Remark \ref{elpoli}, hence it is irreducible. Factoring $Q(x_1^2,x_2^2)$ we get $F_2^*C_{3,k}^{(v)}=\sum_{\bar{\epsilon}\in G_3}C_{\bar{\epsilon}}^{2k}$. By Lemma \ref{tipo4} we obtain
$$(F_2\circ \rho_{n,2k})^*C^{(v)}_{3,k}=\sum_{\epsilon_2,\epsilon_3\in \{\pm 1\}}\sum_{\epsilon_4,\ldots,\epsilon_n\in \{\pm 1\}}C^{2k}_{\epsilon_2,\ldots,\epsilon_n}=\sum_{\bar{\epsilon}\in G_n}C^{2k}_{\bar{\epsilon}}.$$
By definition $C_{n,k}^{(v)}=\rho_{n,k}^*C_{3,k}^{(v)}$, so by the previous commutative diagram 
\begin{equation}\label{fn}
F_n^*C_{n,k}^{(v)}=\sum_{\bar{\epsilon}}C^{2k}_{\bar{\epsilon}},
\end{equation}
in particular $C_{n,k}^{(v)}$ is reduced by Lemma \ref{jeloucito}.

Since $k$ is odd, we see that for all $\bar{\epsilon}\in G_n$ we have 
\begin{equation}\label{fg}
\bar{\epsilon}\cdot C_{(1,\ldots,1)}^{2k}=C_{\bar{\epsilon}}^{2k}
\end{equation}
so $C_{n,k}^{(v)}=F_n(C_{\bar{\epsilon}}^{2k})$ for all $\bar{\epsilon}\in G_n$, hence $C_{n,k}^{(v)}$ is irreducible.

Finally, all the curves $C_{\bar{\epsilon}}\in G_n$ are isomorphic by Equation \eqref{fg}, and the degree of $F_{n|C_{\bar{\epsilon}}^{2k}}:C_{\bar{\epsilon}}^{2k}\to C_{n,k}^{(v)}$ is one, because  $\mathrm{deg}(F_n)=2^{n-2}$, which is exactly the number of curves in the right hand side of Equation \eqref{fn}.
\end{proof}

We finally obtain

\begin{lemma}\label{sdirred2}
%The pullbacks under $\rho_{n,k}$ of curves of type (i) with $\alpha\neq 4,\ldots,n$, of type (ii) and of type (iii) are irreducible and reduced. The pullback of curves of type (i) with $\alpha=4,\ldots,n$ is $\rho_{n,k}^*(C_\alpha)=kR_\alpha$, where $R_\alpha$ are irreducible and reduced. 
If $k$ is even, then the pullbacks of the curves of type (iv) is the union of the curves $C^k_{\bar{\epsilon}}$ with $\bar{\epsilon}\in G$ which are irreducible and reduced. If $k$ is odd, the pullback of the curve of type (v) is irreducible and reduced. 
%5Let $C$ be a curve of type (i), (ii) or (iii) in Lemma \ref{sdomega}. Then $(\rho_{n,k}^*(C))_{\mathrm{red}}$ is smooth and irreducible for every $n\geq 4$.
\end{lemma}

\begin{proof} 
%The pullbacks of the curves of type (i) with $\alpha\neq 4,\ldots n$, of type (ii) and of type (iii) are smooth and complete intersections by Lemma \ref{sdrediii} and Lemma \ref{sdredi}. %From the proofs of these Lemmas, we know that the pullbacks $\rho_{n,k}^*(C_\alpha)$ of curves of type (i) (with $\alpha\neq 4,\ldots,n$), (ii) and (iii) are smooth complete intersections. 
%From \cite{H}, Ex.II.8.4(c) we have that these curves are connected, and therefore irreducible.
%
%By Lemma \ref{lacosa} we have that $$\rho_{n,k}^*C_i=(\pi_{i+1}\circ \cdots\circ \pi_n)^*\rho_{i,k}^*C_i=k(\pi_{i+1}\circ\cdots\circ \pi_n)^*R_i$$ when $4\leq i\leq n$. Since the pullbacks $(\pi_{i+1}\circ\cdots\circ\pi_n)^*R_i$ are smooth and complete intersections by Lemma \ref{sdred2}, we obtain that they are irreducible by \cite{H}, Ex.II.8.4(c). Thus $(\rho_{n,k}^*(C_i))_{\mathrm{red}}=(\pi_{i+1}\circ\cdots\circ\pi_n)^*R_i$ is smooth and irreducible.
%
From Lemma \ref{tipo4}, we know that the curves of type (iv) are irreducible and reduced. By Lemma \ref{jeloucito}, we have then that the curve of type (v) is irreducible and reduced.
\end{proof}

\section{The genus of integral curves on $X_{n,k}$}

Recall that $\rho_{n,k}=\pi_4\circ\cdots\circ\pi_n$. Using the notation $\rho_{n,k}^\bullet$ defined in Section \ref{explanation}, let
$$\omega_n=\rho_{n,k}^\bullet\omega\in H^0(X_{n,k},\mathcal{O}_{X_{n,k}}(2k+3)\otimes S^2\Omega^1_{X_{n,k}/\mathbb{C}}),$$
using the fact that $\rho_{n,k}^*\mathcal{O}_{\mathbb{P}^2}(1)=\mathcal{O}_{X_{n,k}}(1)$.

\begin{lemma}\label{sdomegaxn}
Let $k\geq 2$ and $n\geq 4$. The following curves on $X_{n,k}$ are irreducible and $\omega_n$-integral, and moreover, every $\omega_n$-integral curve is one of these curves: 
\begin{itemize}
\item[(a)] $\rho_{n,k}^*C_{\alpha}$, with $\alpha\in\mathbb{C}\backslash\left\{1,\cdots,n\right\}$. They have genus $\frac{1}{2}k^{n-2}(n(k-1)-2k)+1$. %$2^{n-3}(n-4)+1$;
\item[(a')] $(\pi_{i+1}\circ\cdots\circ\pi_{n})^*R_i=(\rho_{n,k}^*(C_i))_{\mathrm{red}}$, with $4\leq i \leq n$. They have genus $\frac{1}{2}k^{n-3}(n(k-1)-3k+1)+1$. %$2^{n-4}(n-5)+1$.
\item[(b)] $\rho_{n,k}^*C_{\infty}$. It has genus $\frac{1}{2}k^{n-2}(n(k-1)-2k)+1$. %$2^{n-3}(n-4)+1$.
\item[(c)] The pullbacks under $\rho_{n,k}$ of the coordinate axes $x_1=0$, $x_2=0$, $x_3=0$ of $X_{3,k}=\mathbb{P}^2$. They have genus $\frac{1}{2}k^{n-3}(n(k-1)-3k+1)+1$. %$2^{n-4}(n-5)+1$;
\item[(d)] For $k$ even, the irreducible components of the pullbacks of the curves $x_1^{k/2}\pm 2x_2^{k/2}\pm x_3^{k/2}=0$. They have genus $\frac{1}{2}\left(\frac{k}{2}\right)^{n-2}\left(\frac{nk}{2}-k-n\right)+1$. %These are the $2^{n-1}$ curves $[s+t:\pm(2s+t):\cdots :\pm(ns+t)]$ with $[s:t]\in\mathbb{P}^1$ and have genus $0$.
\item[(e)] For $k$ odd, $C^{(v)}_{\epsilon_2,\ldots,\epsilon_n}$. It has genus $\frac{1}{2}k^{n-2}(n(k-1)-2k)+1$.
\end{itemize}
\end{lemma}

\begin{proof} Let $C\subset X_{n,k}$ be an irreducible curve. By Theorem \ref{bullet}, the curve $C$ is $\omega_n$-integral if and only if its image $D=\rho_{n,k}(C)$ is $\omega$-integral, that is, $D$ is one of the curves of type (i)-(v) of Proposition \ref{sdomega}. From this we know that $C$ is an irreducible component of $\rho_{n,k}^*D$, hence $C$ is one of the curves of type (a), (a'), (b), (c), (d) or (e), by Lemmas \ref{sdirred} and \ref{sdirred2}.

Now we compute the genus of these curves.
We have from Ex.IV.3.3.2 in \cite{H} that $\deg_Ci^*\mathcal{O}_{\mathbb{P}^n}(1)=\deg(C)$, which is equal to $\prod_{i=1}^{n-2}d_i$ with $d_i$ the degrees of the equations defining $C$, by \cite{EH} Theorem III-71. From \cite{H}, Ex.II.8.4.(e) we have that for curves of type (a), (a'), (b), (c) and (d) the canonical sheaf is $K_C=\mathcal{O}(\sum_{i=1}^{n-2}{d_i}-(n-1)-1)$. Therefore the geometric genus of $C$ in the cases (a), (b), (c) and (d) is $$g(C)=\frac{1}{2}\left(\prod_{i=1}^{n-2}d_i\right)\left(\sum_{i=1}^{n-2}d_i-n\right)+1$$ because $C$ is smooth in these cases. %From this we can compute the genus of curves of type (a), (a'), (b), (c) and (d).
From Lemma \ref{jeloucito}, we know the genus of curves of type (e), because $C^{(v)}_{\epsilon_2,\ldots,\epsilon_n}$ is birational to a curve $C^{2k}_{\epsilon_2,\ldots,\epsilon_n}$.
%If $D$ is of type (iv), then by Lemma \ref{tipo4} the components of $D$ are the curves of type (d), which are smooth and irreducible of genus $0$ by Lemma \ref{tipo4}. % by Riemann-Hurwitz. 
%
%Now suppose that $D$ is of type (i), (ii) or (iii). Then $(\rho_{n,k}^*(D))_{\mathrm{red}}$ is irreducible by Lemma \ref{sdirred}.
%
%If $D$ is a curve of type (i), then $C=(\rho_{n,k}^*D)_{\mathrm{red}}=\rho_{n,k}^*D$ is of type (a). From Proposition \ref{sdredi} we know the equations defining $C$, thus by Ex.II.8.4(e) in \cite{H}, we have $K_C=\mathcal{O}(2(n-2)-n)=\mathcal{O}(n-4)$, and thus the genus of $C$ is $2^{n-3}(n-4)+1$. Similarly, if $C$ is a curve of type (a'), we have $K_C=\mathcal{O}(n-5)$ by Proposition \ref{sdred2}, hence the genus of $C$ is $2^{n-4}(n-5)+1$.
%If $C$ is the curve of type (b), then by Proposition \ref{sdredi} we have $K_C=\mathcal{O}(n-4)$ and its genus is $2^{n-3}(n-4)+1$.
%If $C$ is a curve of type (c), then by Proposition \ref{sdrediii} $K_C=\mathcal{O}(n-5)$ and $g(C)=2^{n-4}(n-5)+1$.
\end{proof}

\section{Curves of low geometric genus on $X_{n,k}$}

Now we will prove that all the curves in $X_{n,k}$, whose geometric genus is bounded by a certain explicit constant (depending on $n$ and $k$) must be $\omega_n$-integral, and hence they are of type (a), (a'), (b), (c), (d) or (e).

\begin{lemma}\label{losprima}
For each $n\geq 4$, the section $\omega_n\in H^0(X_{n,k},\mathcal{O}_{X_{n,k}}(2k+3)\otimes S^2\Omega^1_{X_{n,k}})$ determines a unique section $\omega_n'\in H^0(X_{n,k},\mathcal{O}_{X_{n,k}}(n(1-k)+5k)\otimes S^2\Omega^1_{X_{n,k}})$. Moreover, all $\omega_n'$-integral curves are $\omega_n$-integral curves.
\end{lemma}

\begin{proof} Let $i=3,\ldots, n$, and let $R_{n,i}=(\pi_{i+1}\circ\cdots\circ\pi_n)^*R_i$.
By Lemma \ref{lacosa}, we have that $\rho_{i,k}(C_i)=kR_i$. Since the curves $C_i$ with $4\leq i\leq n$ are $\omega$-integral, we have by Proposition \ref{elteo1} that $\omega_n$ vanishes identically along $(k-1)\sum_{i=4}^n R_{n,i}$. Since $R_{n,i}$ is the intersection of $X_{n,k}$ and $\left\{x_n=0\right\}$, it is a hyperplane section in $X_{n,k}\subseteq\mathbb{P}^{n-1}$, and hence its ideal sheaf is $\mathcal{O}(-1)$. Thus by Proposition \ref{laprop2}, we get that for each $n$, the section $\omega_n\in H^0(X_{n,k},\mathcal{O}_{X_{n,k}}(2k+3)\otimes S^2\Omega^1_{X_{n,k}})$ determines a unique section $\omega_n'\in H^0(X_{n,k},\mathcal{O}_{X_{n,k}}(2k+3-(k-1)(n-3))\otimes S^2\Omega^1_{X_{n,k}})$ which makes $\omega_n'$-integral curves to be $\omega_n$-integral.
\end{proof}
%For $4\leq i\leq n$, the curves $C_i$ are $\omega$-integral and ramified with respect to $\rho_{n,k}$ by Proposition \ref{sdomegaxn}. The section $\omega_n$ vanishes along $(\pi_{i+1}\circ\cdots\circ\pi_{n})^*R_i=(\rho_{n,k}^*(C_i))_{\mathrm{red}}$ by Proposition \ref{lemavojta} (which is irreducible by Lemma \ref{sdirred}).
%
%Since $(\pi_{i+1}\circ\cdots\circ\pi_n)^*R_i$ is the intersection of $X_{n,k}$ and $\left\{x_n=0\right\}$, it is a hyperplane section in $X_{n,k}\subseteq\mathbb{P}^{n-1}$, and hence its ideal sheaf is $\mathcal{O}(-1)$.
%
%Looking at stalks, we get that $\omega_n$ vanishes along $\sum_{i=3}^n(\pi_{i+1}\circ\cdots\circ\pi_n)^*R_i$. By Proposition \ref{nuevoomega}, there exists a section $\omega_n'\in H^0(X_{n,k},\mathcal{O}(10-n)\otimes S^2\Omega^1_{X_{n,k}})$ such that the $\omega_n'$-integral curves are $\omega_n$-integral curves. 
%\end{proof}

\begin{lemma}\label{sdin}
Let $n>\frac{4\max\{g,1\}+1}{k-1}+5$, and let $C\subseteq X_{n,k}$ be an irreducible curve of genus $g$. Then $C$ is $\omega_n$-integral, and hence is one of the curves of type (a), (a'), (b), (c) or (d).
\end{lemma}

\begin{proof} For $n>\frac{4\max\{g,1\}+1}{k-1}+5$, and an irreducible curve $C\subseteq X_{n,k}$ of genus $g$, let $\varphi_C:\tilde{C}\to X_{n,k}$ be the normalization of $C$. We will first prove that $C$ is $\omega'_n$-integral.
%We know from Example IV.3.3.2 in \cite{H} that $\deg_{\tilde{C}}\varphi_C^*\mathcal{O}(1)=\deg(C)>0$. Thus, since $n>\frac{5}{k-1}+5$, we have that %H is ample, C effective, Nakai Moishezon.
Since $n>\frac{5}{k-1}+5$, we obtain

\begin{eqnarray*}
\deg_{\tilde{C}}(\varphi_C^*\mathcal{O}(n(1-k)+5k)\otimes S^2\Omega^1_{\tilde{C}/\mathbb{C}})&=&\deg_{\tilde{C}}(\varphi_C^*\mathcal{O}(n(1-k)+5k))+2(2g-2)\cr
%&=&(n(1-k)+5k)\deg_{\tilde{C}}\varphi_C^*\mathcal{O}(1)+4g-4\cr
&=&(n(1-k)+5k)\deg(C)+4g-4\cr
&<&n(1-k)+5k+4g-4<0
\end{eqnarray*}
Hence $H^0(\tilde{C},\varphi_C^*\mathcal{O}(n(1-k)+5k)\otimes S^r\Omega^1_{\tilde{C}/\mathbb{C}})=0$, and so $\varphi_C^*\omega'_n=0$. From this, the curve $C$ must be $\omega_n'$-integral. By Lemma \ref{losprima}, we get that $C$ is $\omega_n$-integral.
The last statement holds by Lemma \ref{sdomegaxn}.
\end{proof}

\section{Proof of the main results}\label{sdproofmain}

\begin{proof}[Proof of Theorem \ref{main}]
From Lemma \ref{sdin} we have that all the curves of genus $g$ in $X_{n,k}$ are $\omega_n'$-integral. Hence by Lemma \ref{losprima}, these curves are $\omega_n$-integral.

Note that when $n>\frac{4g+1}{k-1}+4$ (with $g\geq 1$) we have
\begin{eqnarray*}
\frac{1}{2}k^{n-2}(n(k-1)-2k)+1&>&\frac{k^{n-2}}{2}\left(\left(\frac{4g+1}{k-1}+4\right)(k-1)-2k\right)+1\cr
&=&\frac{k^{n-2}}{2}(4g+2k-4)+1\geq 4g+4-4+1\geq g.
\end{eqnarray*}
hence curves of type (a) and (b) have genus greater than $g$.

We also have
\begin{eqnarray*}
\frac{1}{2}k^{n-3}(n(k-1)-3k+1)+1&>&\frac{k^{n-2}}{2}\left(\left(\frac{4g+1}{k-1}+4\right)(k-1)-3k+1\right)+1\cr
&=&\frac{k^{n-2}}{2}(4g+k-2)+1\geq g.
\end{eqnarray*}
So curves of type (a') and (c) have genus greater than $g$.

If $k=2$, we have
\begin{eqnarray*}
\frac{1}{2}\left(\frac{k}{2}\right)^{n-2}\left(\frac{nk}{2}-k-n\right)+1&=&\frac{1}{2}(n-2-n)+1=0.
\end{eqnarray*}
Thus when $k=2$, the curves of type (d) have genus $0$.

However, if $k\geq 3$
\begin{eqnarray*}
\frac{1}{2}\left(\frac{k}{2}\right)^{n-2}\left(\frac{nk}{2}-k-n\right)+1&\geq &\frac{1}{2}\left(\frac{k}{2}\right)^{n-2}\left(n\left(1+\frac{k}{2}\cdot\frac{k-2}{k}\right)-k-n \right)+1\cr
&=&\frac{1}{2}\left(\frac{k}{2}\right)^{n-2}\left(\frac{nk}{2}\cdot\frac{k-2}{k}-k\right)+1\cr
&>&\frac{1}{2}\left(\frac{k}{2}\right)^{n-2}\left(\frac{\left(\frac{4g+1}{k-1}+4\right)k}{2}\cdot\frac{k-2}{k}-k\right)+1\cr
&=&\frac{1}{2}\left(\frac{k}{2}\right)^{n-2}\left( \frac{4g+1}{2}\cdot\frac{k-2}{k-1}+k-4\right)+1\cr
&\geq & \frac{27}{16}\left(\frac{4g+1}{2}\cdot \frac{1}{2}+3-4\right)+1\cr
&=&\frac{27}{16}\left(g-\frac{3}{4}\right)+1>g
%&=&\left(\frac{k}{2}\right)^{n-3}\frac{k}{k-1}\frac{4g+1}{4}+1\geq g
\end{eqnarray*}
%\begin{eqnarray*}
%\frac{1}{2}\left(\frac{k}{2}\right)^{n-2}\left(\frac{nk}{2}-k-n\right)+1&\geq &%\frac{1}{2}\left(\frac{k}{2}\right)^{n-2}\left(3\left(\frac{n}{2}\right)-1-n\right)+1\cr
%%&=&
%\frac{1}{2}\left(\frac{k}{2}\right)^{n-2}\left(\frac{n}{2}-3\right)+1\cr
%&>&\frac{1}{2}\left(\frac{k}{2}\right)^{n-2}\left(\frac{4g+1}{k-1}+4-4\right)+1\cr
%&=&\left(\frac{k}{2}\right)^{n-3}\frac{k}{k-1}\frac{4g+1}{4}+1\geq g
%\end{eqnarray*}
so curves of type (d) when $k\geq 3$ have genus greater than $g$. From this calculation, we also see that the curve of type (e) for $k\geq 3$ odd has genus greater than $g$.

%Since $g<\frac{n-6}{4}$ and $n>9$, we know by Lemma \ref{sdomegaxn} that the only $\omega_n$-integral curves of genus less than or equal to $g$ are the curves of type (d).
From this we obtain that if $k=2$, then the only curves of genus less than or equal to $g$ in $X_{n,k}$ are the curves of type (d), and if $k\geq 3$, then there are no curves of genus less than or equal to $g$ on $X_{n,k}$.
\end{proof}

\begin{proof}[Proof of Theorem \ref{main2}] Let $C$ be a curve of genus $0$ or $1$. If $k=2$, then by Theorem \ref{main}, we have that for $n\geq 11$ all the curves of genus $0$ or $1$ are the curves of type (d). If $k\geq 3$, then by Theorem \ref{main} we have that for $n\geq \frac{5}{k-1}+5$ there are no curves of genus $0$ or $1$ on  $X_{n,k}$.
\end{proof}

\begin{proof}[Proof of Theorem \ref{hyper}]
Similar to Theorem 6.1 in \cite{V}.
\end{proof}

\begin{proof}[Proof of Theorem \ref{sdffield}]  Let $k\geq 2$ and let $n>\frac{4g+1}{k-1}+4$. Let $K$ be a function field of genus $g$, and let $C_K$ be a curve with function field $K$. The solutions over $K$ (up to scaling by elements of $K$) of the system of equations \eqref{xnequation} are in one to one correspondence with the morphisms $\left\{f:C_K\to X/\mathbb{C}\right\}$, in such a way that constant morphisms are in correspondence with sequences of proportional to a sequence of complex numbers.
By Riemann-Hurwitz, these morphisms are either constant, or must map the curve $C_K$ to curves in $X$ with genus less than or equal to $g$. By Theorem \ref{main}, the only curves with genus less than or equal to $g$ are the curves of type (d), and only in the case that $k=2$. The result follows from Theorem \ref{main}
%Hence for $k\geq 3$, the sequence $(f_1,\ldots,f_n)$ must be proportional to a sequence of complex numbers.
%If $k=2$, then each non-constant map $f:C_K\to X$ must have image contained in one of the curves of type (d), which means that the corresponding solution $(f_1:\ldots:f_n)$ to the previous system is not proportional to a constant solution, and it corresponds to a $K$-rational point in a curve of type (d). Hence it is of the form $$[f_1:\ldots:f_n]=[\pm(s+t):\ldots:\pm(ns+t)].$$
%Solving for $s$ and $t$ we see that they are in $K$, so we get $a,b\in K$ satisfying $$[f_1:\ldots:f_n]=[\pm(a+b):\ldots:\pm(na+b)].$$ 
\end{proof}

Now we will prove our arithmetic results. For the convenience of the reader, we state the Bombieri-Lang conjecture:

\begin{conjecture}\label{bombieri}
If $X$ is a smooth projective algebraic variety of general type defined over a number field $K$, then there exists a proper Zariski-closed subset $Z$ of $X$ such that the set $X(K)\setminus Z(K)$ is finite.
\end{conjecture}

See \cite{nog} for Bombieri's version in the case of surfaces, and see \cite{lang} for Lang's version in all dimensions (which claims additional uniformity on the number field). In the case of surfaces, the set $Z$ can only be a finite union of curves of genus $0$ or $1$, by Faltings' Theorem (cf.\ \cite{Faltings}).

\begin{proof}[Proof of Theorem \ref{miain}] Let $n\geq 11$ and let $a_n,\ldots,a_1$ be a sequence of $n$ rational numbers whose squares have constant second differences. %, but which is not an arithmetic progression (up to signs), 
Hence for all $4\leq i\leq n$, it satisfies $a_{i-3}^2-3a_{i-2}^2+3a_{i-1}^2=a_i^2.$

If $k=2$ and the sequence $a_n,\ldots,a_1$ is not in arithmetic progression (up to signs), then there are no $s,t\in\mathbb{Z}$ such that $a_i^2=(si+t)^2$. Thus, the point $[a_1:\ldots:a_n]$ is in $X_{n,2}$, and it is not in one of the curves of type (d). By the Bombieri-Lang conjecture on $X_{n,2}$ and Theorem  \ref{main2} we can only have finitely many of these points.

To deduce an absolute bound (possibly larger than $11$) from the
finiteness, we follow an elementary combinatorial idea of Vojta
\cite{V} adapted to our case.

Suppose that there are, up to scaling, exactly $N$ sequences of
squares $x_1^2,\ldots,x_{12}^2$ with $x_i\in\mathbb{Q}$ having
constant second differences and such that the sequence is non-trivial
(i.e.\ the $x_i$ are not an arithmetic progression up to sign). We
claim that there is no nontrivial sequence of rational $M=N+12$
squares with constant second differences.

Indeed, suppose that $a_1^2,\ldots, a_M^2$ is such a non-trivial
sequence.  Then for all $1\le i\le N+1$ we have that
$a_{i}^2,\ldots,a_{i+11}^2$ also has constant second differences and
we claim that it is non-trivial. Suppose that it is trivial, then in
particular there are signs $\epsilon,\tau\in \{1,-1\}$ with
$a_i -2\epsilon a_{i+1}+\tau a_{i+2}=0$. Note that
$[a_i:a_{i+1}:a_{i+2}]\in X_{3,k}=\mathbb{P}^2$ lies on a curve of type (iv) and
$[a_i:\ldots: a_M ]\in X_{M-i+1}$, so it is in a curve of type (d), so
$a_i,a_{i+2},a_{i+3},\dots, a_M$ is an arithmetic progression up to
sign. Similarly $a_1,\ldots, a_i,a_{i+2},a_{i+3}$ is an arithmetic
progression up to sign (using the same signs for $a_1,a_2,a_3$), and
we get that the sequence $a_1,\ldots,a_M$ is trivial, a contradiction
with the fact that it is non-trivial.

Thus, we obtain non-trivial sequences $a_{i}^2,\ldots,a_{i+11}^2$  for
all $1\le i\le N+1$. We claim that they are non-proportional. Indeed,
there is a polynomial $f(t)=ut^2+vt+w$ such that $f(n)=a_n^2$ for all
$1\le n\le M$ because the $a_n$ have constant second differences, and
our sequence is non-degenerate (because it is non-trivial) so $f$ is non-constant. It is easy to
check that the function $F:\mathbb{A}^1\to \mathbb{P}^2$ defined by $t\mapsto
[f(t):f(t+1):f(t+3)]$ is injective, proving our claim.

Finally, this is a contradiction because there are at most $N$
non-proportional non-trivial sequences of length $12$ and we have
produced $N+1$ of them. This proves that there is no non-trivial
sequence of length $M=N+12$.
\end{proof}

\begin{proof}[Proof of Theorem \ref{hey}]
If $k\geq 3$, and $n>\frac{5}{k-1}+5$ then the point $[a_1:\ldots:a_n]$ is in $X_{n,k}$. By the Bombieri-Lang conjecture and Theorem \ref{main2} we can only have finitely many points of this form in $X_{n,k}$. The rest of the proof is the same as the proof of Theorem \ref{miain}.
\end{proof}

\begin{proof}[Proof of Theorem \ref{moha}]
Let $M\geq 3$ be a positive integer. Let $s=\left\{(x_j,y_j)\right\}_{j=1}^M$ be a $\mathbb{Q}$-rational $y$-arithmetic progression on the curve $y^2=x^k+b$ for some $b\in\mathbb{Q}^*$. Then $x_1,\ldots,x_M$ is a sequence of rational numbers whose $k$-th powers have constant second differences. Let $s'=\left\{(x'_j,y'_j)\right\}_{j=1}^M$ be another $\mathbb{Q}$-rational $y$-arithmetic progression on a curve $y^2=x^k+b'$, with $b'\in \mathbb{Q}^*$. 

We claim that there is $c\in \mathbb{Q}^*$ such that $x'_j=cx_j$ for each $1\le j\le M$ if and only if $s$ and $s'$ are equivalent. After this claim is proved, the result will follow from Theorem 10.

If $s$ and $s'$ are equivalent then, in particular, there is $\lambda\in \mathbb{Q}^*$ such that $x'_j=\lambda x_j$, so we can take $c=\lambda$. 

Conversely, suppose that $x'_j=cx_j$ for each $1\le j\le M$. Since $s$ and $s'$ are $y$-arithmetic progressions, there are $u,v,u',v'\in\mathbb{Q}$ with $v,v'\neq 0$ and such that for each $j$ we have $y_j=u+vj$ and $y'_j=u'+v'j$. Taking second differences of $x_j^k$ and ${x'_j}^k$ we get
$$
2v'^2={x'}^k_{j+2}-2{x'}^k_{j+1}+{x'}^k_{j} = \lambda^k(x_{j+2}^k-2x_{j+1}^k+x_{j}^k)=\lambda^k\cdot 2v^2
$$
and we see that $\lambda^k$ is a square, and there is $\mu\in \mathbb{Q}$ (of appropriate sign) which satisfies $\mu^2=\lambda^k$ and  $v'=\mu v$. Taking the first differences of $x_j^k$ and ${x'_j}^k$ we get
$$
2u'v' + (2j+1){v'}^2 = {x'}^k_{j+1}-{x'}^k_{j} = \mu^2 (  {x}^k_{j+1}-{x}^k_{j}) = \mu^2(2uv + (2j+1){v}^2)
$$
so we obtain $u'=\mu u$, and we get $y'_j = \mu y_j$. Therefore, $s$ and $s'$ are equivalent. 
\end{proof}

\begin{proof}[Proof of Theorem \ref{catorce} and Theorem \ref{diecisiete}]
The proof is similar to that of Theorem \ref{moha}, replacing $\mathbb{Q}$ by the function fields of a curve over $\mathbb{C}$ or $\mathbb{Q}$ respectively, and using Theorem \ref{sdffield} instead of Theorem \ref{hey}.
\end{proof}

\end{document}